\documentclass[a4paper,10pt,reqno]{amsart}
\usepackage[T1]{fontenc}
\usepackage{lmodern}
\usepackage{amsmath,amssymb,amsthm,mathtools,booktabs,array,enumitem,float}
\usepackage[a4paper,textwidth=6.78in,textheight=10.10in,centering,headsep=15pt]{geometry}
\usepackage{microtype}
\usepackage{xcolor}
\usepackage[colorlinks=true,linkcolor=blue!45!black,citecolor=blue!45!black,urlcolor=blue!45!black,
  bookmarksnumbered=true,pdfauthor={K. Srinivasa Raghava},
  pdftitle={Eta quotients at squarefree level: modular units and Heegner values}]{hyperref}
\newtheorem{theorem}{Theorem}[section]
\newtheorem{proposition}[theorem]{Proposition}
\newtheorem{lemma}[theorem]{Lemma}
\newtheorem{corollary}[theorem]{Corollary}
\theoremstyle{definition}
\newtheorem{example}[theorem]{Example}
\theoremstyle{remark}
\newtheorem{remark}[theorem]{Remark}
\numberwithin{equation}{section}
\newcommand{\Z}{\mathbb Z}
\newcommand{\Q}{\mathbb Q}
\newcommand{\C}{\mathbb C}

\newcommand{\Hh}{\mathfrak H}
\newcommand{\Cl}{\mathcal C}
\newcommand{\eps}{\varepsilon}

\DeclareMathOperator{\lcm}{lcm}
\setlist{nosep,leftmargin=*}
\allowdisplaybreaks[1]
\title[Eta quotients and Heegner values]{Eta quotients at squarefree level:\\ modular units and Heegner values}
\author{K. Srinivasa Raghava}
\address{Pie Mathematics Association}
\email{srinivasaraghavak@gmail.com}
\date{}
\AddToHook{cmd/normalsize/after}{%
  \setlength{\abovedisplayskip}{4pt plus 2pt minus 1pt}%
  \setlength{\belowdisplayskip}{4pt plus 2pt minus 1pt}%
  \setlength{\abovedisplayshortskip}{0pt plus 1pt}%
  \setlength{\belowdisplayshortskip}{3pt plus 1pt minus 1pt}%
  \setlength{\jot}{2.5pt}}
\begin{document}
\begin{abstract}
At squarefree level \(M\), each nontrivial character of the group of divisors of \(M\) defines an eta quotient \(R_T\). We determine the quotient of the modular unit group, modulo constants, by the least modular powers of these eta quotients. Its exponent is \(2^{k-2}\) when \(M\) has \(k\ge3\) prime divisors. We also determine the signs in the Atkin--Lehner transformation law and obtain polynomial identities at sixteen levels.

For a quadratic order of discriminant \(D<0\) and conductor prime to \(M\), a root of \(D\) modulo \(4M\) selects prime ideal classes \(P_p\) for \(p\mid M\). The Fourier coefficient of the normalized \(\log|R_T|\) at a class character \(\psi\) equals the corresponding coefficient of \(\log(\sqrt{\Im\tau}\,|\eta(\tau)|^2)\) multiplied by \(\tfrac12\prod_{p\in T}(1-\psi(P_p))\prod_{p\mid M,\,p\notin T}(1+\psi(P_p))\).
For fundamental discriminants and for orders of conductor two in fields of odd discriminant, the normalized absolute values are products of rational powers of real quadratic units whenever the displayed multiplier vanishes for every nonreal character. At \(i/\sqrt{66}\) we obtain such an evaluation for one quotient and a quartic unit with no positive power in an abelian number field for another. Further examples include Ramanujan's class invariants and exact values at class number one.
\end{abstract}
\maketitle

\section{Introduction}
Two eta quotients at the same imaginary quadratic point can have quite different arithmetic.
At \(\tau=i/\sqrt{66}\) one of the products below is a product of rational powers of real quadratic units. Another has a conjugate on the unit circle and no positive power in any abelian extension of \(\Q\).
The signs attached to the divisors of \(66\) account for the difference.
We determine how these signs select the contributing class characters and how they govern the integral structure of modular units.

\subsection{Eta quotients attached to characters}
Let \(M>1\) be squarefree, let \(P\) be its set of primes, and put \(k=|P|\).
For a nonempty \(T\subseteq P\), write
\[
 \chi_T(d)=(-1)^{\#\{p\in T:p\mid d\}},\qquad
 R_T(\tau)=\prod_{d\mid M}\eta(d\tau)^{\chi_T(d)},\qquad
 \Lambda_T=\prod_{p\in T}(p-1)\prod_{p\in P\setminus T}(p+1).
\]
Here \(\eta(\tau)=q^{1/24}\prod_{n\ge1}(1-q^n)\), with \(q=e^{2\pi i\tau}\).
The divisors form a group under \(d*d'=dd'/\gcd(d,d')^2\). Its characters are the \(\chi_T\), including the trivial character for \(T=\varnothing\).
Let \(e_T\) be the least positive integer for which \(R_T^{e_T}\) is a modular function on \(\Gamma_0(M)\). Throughout the paper we use
\[
 e_T=
 \begin{cases}
 24/\gcd(24,\Lambda_T),& k\ge2,\\
 \lcm\bigl(2,24/\gcd(24,p-1)\bigr),&M=p,
 \end{cases}
 \qquad  u_T=R_T^{e_T},\qquad m_T=e_T\Lambda_T/24.
\]
Proposition~\ref{prop:minimal} gives this exponent and the divisor of \(u_T\). The exponent at prime level includes the character condition.
We abbreviate subscripts by their members. For instance, at level \(6\),
\[
 R_2=\frac{\eta(\tau)\eta(3\tau)}{\eta(2\tau)\eta(6\tau)},\quad
 R_3=\frac{\eta(\tau)\eta(2\tau)}{\eta(3\tau)\eta(6\tau)},\quad
 R_{2,3}=\frac{\eta(\tau)\eta(6\tau)}{\eta(2\tau)\eta(3\tau)}.
\]
Up to a constant, the nonzero integral powers of the \(u_T\) are exactly the nonconstant units on \(X_0(M)\) preserved up to a constant and inversion by every Atkin--Lehner involution (Remark~\ref{rem:eigenunits}).

\subsection{Integral structure}
Let \(Q_M\) be the group of modular units on \(X_0(M)\) modulo the subgroup generated by constants and the \(u_T\).
Write \(\kappa_2=1\) if \(2\mid M\), and \(\kappa_4=1\) if a prime \(p\equiv3\bmod4\) divides \(M\), with value zero otherwise.
\begin{theorem}\label{intro-integral}
The group \(Q_M\) is a finite \(2\)-group. For odd \(M\) with \(k\ge3\) and even \(M\) with \(k\ge4\), its cyclic summands have the following multiplicities. Omit zero multiplicities and summands of order one.
\[
\begin{array}{c|c}
 \textup{order}&\textup{multiplicity}\\ \hline
 2^{k-2}& k+\binom{k}{2}-\kappa_2-\kappa_4\\
 2^{k-3}& \binom{k}{3}+\kappa_4\\
 2^{k-4}\ (k\ge4)&\binom{k}{4}+\kappa_2\\
 2^{k-s}\ (5\le s\le k-1)&\binom{k}{s}
\end{array}
\]
For \(k=1\), and for odd \(M\) with \(k=2\), \(Q_M=0\). For \(M=2p\), its isomorphism types for \(p\equiv1,3,5,7\bmod8\) are respectively \(0,(\Z/2)^2,(\Z/2)^2,\Z/2\). For even \(M\) with \(k=3\), it is \((\Z/2)^5\).
\end{theorem}
Thus \(2^{k-2}\) is the sharp common exponent for expressing powers of all units through the \(u_T\) when \(k\ge3\).
Section~3 relates \(Q_M\) to the cuspidal divisor class group \(C(M)\) by
\(\lvert C(M)\rvert\lvert Q_M\rvert=2^{k(2^{k-1}-1)}\prod_{T\ne\varnothing}m_T\) for \(k\ge2\).
Proposition~\ref{prop:al-law} computes the multiplier in \(u_T\circ W_Q\) itself. Its sign can be nontrivial for two odd primes, where it is an explicit power of a Legendre symbol.

\subsection{Characters of the form class group}
Let \(D<0\) be the discriminant of a quadratic order, with conductor prime to \(M\), and fix \(\beta\) with \(\beta^2\equiv D\bmod4M\).
We call such \(D\) admissible for \(M\).
Write \(\Cl(D)\) for the form class group, of order \(h(D)\). Each class \(X\) has a primitive positive definite representative \([A,B,C]\) with \(M\mid A\) and \(B\equiv\beta\bmod2M\).
Put \(\tau_X=(-B+\sqrt D)/(2A)\), taking \(\Im\sqrt D>0\), and let \(P_p\) be the class of \([p,\beta,(\beta^2-D)/(4p)]\).
The function \(\lambda(\tau)=\log(\sqrt{\Im\tau}\,|\eta(\tau)|^2)\) depends only on the form class, so write \(S(\psi)=\sum_X\psi(X)\lambda(X)\).

\begin{theorem}[Fourier formula for Heegner values]\label{intro-character}
For every character \(\psi\) of \(\Cl(D)\) and every nonempty \(T\),
\[
 \sum_X\psi(X)\bigl(\log|R_T(\tau_X)|-\delta_T\bigr)
 =\tfrac12 S(\psi)\prod_{p\in T}(1-\psi(P_p))
                         \prod_{p\in P\setminus T}(1+\psi(P_p)),
\]
where \(\delta_{\{p\}}=2^{k-3}\log p\) and \(\delta_T=0\) otherwise.
\end{theorem}
Thus the Fourier coefficients of the normalized logarithms depend on the primes of the level through their ideal classes. For \(T=\{p\}\) the normalization divides the absolute value by \(p^{2^{k-3}}\).
For fundamental \(D\), a nontrivial real class character corresponds to a factorization \(D=D_1D_2\) into fundamental discriminants with \(D_1>1\).
Put \(U(\psi)=\eps_{D_1}^{w(D)h(D_1)h(D_2)/w(D_2)}\), where \(\eps_{D_1}>1\) is the fundamental unit, \(h\) is the ordinary class number, and \(w\) counts roots of unity in the order. Thus \(w(-3)=6,\ w(-4)=4\), and \(w(D)=2\) otherwise.
For a real character put \(T(\psi)=\{p\in P:\psi(P_p)=-1\}\). We take all powers of positive units in their indicated real embedding.
\begin{theorem}[Evaluation by genus characters]\label{intro-selection}
Suppose \(D\) is fundamental. The \(\psi\)-twisted product of \(|R_T(\tau_X)|\) for a nontrivial genus character equals \(U(\psi)^{-2^{k-1}}\) if \(T=T(\psi)\) and equals \(1\) otherwise.
If the multiplier in Theorem~\ref{intro-character} vanishes for every nonreal character, then
\[
 |R_T(\tau_X)|=e^{\delta_T}
 \prod_{\substack{\psi\ne1\\T(\psi)=T}} U(\psi)^{-\psi(X)2^{k-1}/h(D)},
\]
where the product runs over genus characters. This hypothesis holds for every \(T\) if the class group has exponent at most two.
\end{theorem}
Theorem~\ref{thm:hhclosed} also treats \(D=4D_0\) for odd fundamental \(D_0\), with the explicit conductor correction in Proposition~\ref{prop:hhkronecker}. These formulas give absolute values. For class number one, Theorem~\ref{thm:hhclass1} determines \(R_T(\tau_X)^{24}\) and places \(u_T(\tau_X)\) in \(\Q(\sqrt D)\).

\subsection{Examples}
At \(M=10\), \(\tau=(-3+i)/10\), and at \(M=22\), \(\tau=(-9+\sqrt{-7})/44\), respectively, we obtain
\begin{align}
 (R_2,R_5,R_{2,5})(\tau)&=(1+i,\,1+2i,\,1),\label{intro-gauss}\\
 (R_2,R_{11},R_{2,11})(\tau)&=((1+\sqrt{-7})/2,\,-2+\sqrt{-7},\,1).\label{intro-seven}
\end{align}
For \(M=41\cdot43\cdot47\) and \(\tau=(-21239+\sqrt{-163})/(2M)\),
\begin{equation}
 R_{41}=\Bigl(\frac{1+\sqrt{-163}}2\Bigr)^2,\quad
 R_{43}=\Bigl(\frac{-3+\sqrt{-163}}2\Bigr)^2,\quad
 R_{47}=\Bigl(\frac{-5+\sqrt{-163}}2\Bigr)^2 .
\end{equation}
Here \(R_T^{24}=1\) for \(|T|\ge2\).
For discriminant \(-24\), take distinct primes \(p,q\) represented by \(2x^2+3y^2\) (for instance \(2,3,5,11,29,53\)), choose \(\beta^2\equiv-24\bmod4pq\) and put \(\tau=(-\beta+\sqrt{-24})/(2pq)\). Then
\begin{equation}
 \left|\frac{\eta(\tau)\eta(pq\tau)}{\eta(p\tau)\eta(q\tau)}\right|
 =(1+\sqrt2)^{-1/3}.
\end{equation}
With \(\phi=(1+\sqrt5)/2\), the corresponding answers for \(2x^2+xy+2y^2\) of discriminant \(-15\) and \(2x^2+2xy+3y^2\) of discriminant \(-20\) are \(\phi^{-1/3}\) and \(\phi^{-1/2}\). At level \(6\), \(R_{2,3}\) is positive on \(|\tau|=1/\sqrt6\) by the argument of Example~\ref{ex:hhcubic}. It takes these values at \((-3+\sqrt{-15})/12\) and \((-1+\sqrt{-5})/6\), respectively. At discriminant \(-23\),
\begin{equation}
 R_{2,3}\bigl((-1+\sqrt{-23})/12\bigr)=1/\theta,\qquad
 \theta^3=\theta+1,\quad \theta>1.
\end{equation}
Write \(\varpi_T(M)=R_T(i/\sqrt M)\) and, for odd \(M\), \(\varpi'_T(M)=|R_T((1+i/\sqrt M)/2)|\).
With \(f(-q)=\prod_{n\ge1}(1-q^n)\) and \(q=e^{-2\pi/\sqrt{210}}\),
\begin{equation}
 \prod_{d\mid210}f(-q^d)^{\mu(d)}=e^{4\pi/\sqrt{210}}\phi^{-4}.
\end{equation}
For squarefree \(M\) with odd \(k\ge3\), the corresponding product at \(q=e^{-2\pi/\sqrt M}\) is \(q^{\varphi(M)/24}\).
The example in the opening paragraph has \(\Cl(-264)\simeq\Z/4\oplus\Z/2\). For \(T=\{2,11\}\) the coefficients of the nonreal characters vanish, giving
\begin{equation}
 \varpi_{2,11}(66)=(\sqrt3+\sqrt2)^{-1}(7\sqrt2+3\sqrt{11})^{-1/3}.
\end{equation}
For \(T=\{2,3\}\) the value satisfies \(x^4-3x^3-4x^2-3x+1=0\). Sections~6 and 7 give both evaluations and explain their different conjugates.
The second Fricke point gives, for example,
\begin{equation}
 \varpi_{3,7}(21)\varpi'_{3,7}(21)=2-\sqrt3,\qquad
 \frac{\varpi_{3,7}(21)}{\varpi'_{3,7}(21)}
 =\Bigl(\frac{5+\sqrt{21}}2\Bigr)^{-1/2}.
\end{equation}
For \(\ell\mid M\) prime and \(n=M/\ell\), put \(\rho_\ell=\ell^{-1/4}\eta(i\sqrt{n/\ell})/\eta(i\sqrt{n\ell})\). Fourier inversion gives
\begin{equation}
 \rho_\ell^{\,2^{k-1}}=\prod_{\substack{\ell\in T\\ |T|\ {\rm even}}}\varpi_T(M)^{-1}.
\end{equation}
For \(M=\ell\) both sides are one by eta inversion. At \(\ell=2\) this gives Ramanujan's \(g_M\). At \(M=3n\), \(n>1\) squarefree and prime to \(6\), the second point gives \(\lambda_n=\rho'_3(3n)^6\). The appendix covers the 30 known first-point levels and the 39 known second-point levels with one class per genus in the ranges of Section~6.

\subsection{Trace identities}
Put \(c_T=p^{e_T2^{k-2}}\) for \(T=\{p\}\), and \(c_T=1\) otherwise.
\begin{theorem}[Trace identities]\label{intro-trace}
For composite squarefree \(M\), some \(\Lambda_T\) divides \(24\) exactly at the levels
\[
 6,10,14,15,21,22,26,30,33,35,39,42,46,66,70,78 .
\]
At each of these levels the full Atkin--Lehner quotient has a Hauptmodul \(j_M=q^{-1}+O(q)\).
When the sign character of \(u_T\) is trivial, \(u_T+c_Tu_T^{-1}\) is a monic polynomial in \(j_M\) of degree \(m_T\). For odd \(M=pq\), if \(T=\{p,q\}\) has nontrivial sign, the same holds for \(u_T^{-1}-u_T\). In either case the invariant for \(m_T=1\) is \(j_M+(-1)^{|T|}e_T\).
\end{theorem}
For example, \(u_{3,5}^{-1}-u_{3,5}=j_{15}+3\) at level \(15\). At level \(35\) the signed difference supplies the simple pole. At level \(30\), writing \(j=j_{30}\), one obtains
\begin{equation}
 j=R_2+4/R_2+1
  =R_{3,5}+R_{3,5}^{-1}-1
  =R_{2,3}^2+R_{2,3}^{-2}-2
  =R_{2,3,5}^3+R_{2,3,5}^{-3}+3 .
\end{equation}
At \(i/\sqrt M\), \(j_M\) takes the integer values \(22,12,8,6,5,4,3,3\) for \(M=6,10,15,22,30,42,70,78\), respectively. Corollary~\ref{cor:fricke-traces} gives all sixteen values.

\subsection{Relation to earlier work}
The description of squarefree-level modular units by eta quotients and the cusp-order formula are classical \cite{newman,ligozat,takagi,kohnen}. Takagi computed the squarefree cuspidal class number, and Yoo determined the rational cuspidal divisor class group for arbitrary level \cite{takagi,yoo}. Takagi's study of Atkin--Lehner quotients already displays a Legendre-symbol dependence at two odd primes and diagonalizes the cusp operator by divisor-group characters \cite[\S8]{takagi2010}.
Proposition~\ref{prop:al-law} gives the Atkin--Lehner action on \(R_T^{e_T}\), including its sign. Raising this identity to the power \(24/e_T\) gives the action on \(R_T^{24}\). Theorem~\ref{thm:lattice} determines the invariant factors of \(Q_M\), and Theorem~\ref{thm:two-fricke}(c) evaluates the first Fricke values when the class group has exponent at most two.

Kronecker's method for \(g_n,G_n,\lambda_n\) is developed in \cite{bcz,bckz}. Enge and Schertz treat the multiple eta quotient corresponding to \(T=P\) and its class invariants \cite{enge}. For a genus character in Byeon's Heegner setting, the multiplier for an eta exponent vector \(r\) is \(\sum_{d\mid M}\psi(P_d)r_d\), in the normalization using \(\eta^{24}\) \cite[Proposition~3.1]{byeon}. Taking \(r_d=\chi_T(d)\) factors this sum into the local signs of Theorem~\ref{intro-character}. Our formulation treats all class characters and all admissible quadratic orders, gives a sufficient condition for evaluating an individual absolute value by genus characters, and yields signed trace identities and evaluations at both Fricke points. For simple poles the trace formulas include McKay--Thompson series and Ramanujan's \(P\)--\(Q\) equations \cite{berndt4,conway,oeis}.

\section{Eta quotients and Atkin--Lehner involutions}\label{sec:walsh}

\subsection{Cusp orders and modularity}
A modular unit is a nonzero meromorphic modular function whose zeros
and poles are supported at the cusps. At squarefree level we represent
the cusps by \(1/c\) with \(c\mid M\), including \(\infty\sim 1/M\).
For \(r=(r_d)_{d\mid M}\in\Z^{2^k}\) put
\(f_r(\tau)=\prod_{d\mid M}\eta(d\tau)^{r_d}\).
Let \(I_M=\{r:\sum_d r_d=0\}\), and let \(K_M\) be the sublattice of
\(I_M\) defined by
\begin{equation}\label{eq:eta-criterion}
 \sum_{d\mid M}d r_d\equiv\sum_{d\mid M}(M/d)r_d\equiv0\pmod {24},
 \qquad \sum_{\substack{d\mid M\\p\mid d}}r_d\equiv0\pmod2
 \quad(p\mid M,\ p\text{ prime}).
\end{equation}

\begin{theorem}[Newman, Ligozat, Takagi]\label{thm:eta-criterion}
For \(r\in I_M\), the function \(f_r\) is modular on \(\Gamma_0(M)\)
if and only if \(r\in K_M\). In that case
\begin{equation}\label{eq:cusp-orders}
 \operatorname{ord}_{1/c}(f_r)
 =\frac1{24}\sum_{d\mid M}\frac{M\gcd(c,d)^2}{cd}\,r_d
 \qquad(c\mid M).
\end{equation}
Every modular unit on \(X_0(M)\) is a constant times \(f_r\) for a
unique \(r\in K_M\).
\end{theorem}
The criterion and cusp orders are due to Newman and Ligozat \cite{newman,ligozat}.
For the integral exponents of arbitrary modular units see Takagi \cite{takagi}
or Wang and Yang \cite[Theorem~2, \(n=1\)]{wangyang}. In the latter
the integer shift denominator satisfies \(h^2\mid M\), so \(h=1\).
Write \(v_T=(\chi_T(d))_{d\mid M}\).
The matrix taking \(r\) to \((24\operatorname{ord}_{1/c}f_r)_{c\mid M}\)
is the tensor product of the matrices
\(\left(\begin{smallmatrix}p&1\\1&p\end{smallmatrix}\right)\)
in the coordinates \(p\nmid d,\ p\mid d\). Its eigenvalue on
\(v_T\) is \(\Lambda_T\). We shall also use
\begin{equation}\label{eq:walsh-sums}
 \sum_{d\mid M}d\chi_T(d)=(-1)^{|T|}\Lambda_T,\quad
 \sum_{d\mid M}(M/d)\chi_T(d)=\Lambda_T,\quad
 \sum_{\substack{d\mid M\\p\mid d}}\chi_T(d)=
 \begin{cases}-2^{k-1},&T=\{p\},\\0,&T\ne\{p\}.\end{cases}
\end{equation}

\begin{proposition}\label{prop:minimal}
For any integer \(e\), \(R_T^e\) is modular on \(\Gamma_0(M)\) if and
only if \(e_T\mid e\). With \(m_T=e_T\Lambda_T/24\),
\[
 \operatorname{div}(u_T)=m_T\sum_{c\mid M}\chi_T(c)(1/c),\qquad
 u_T=q^{(-1)^{|T|}m_T}\bigl(1-e_Tq+O(q^2)\bigr).
\]
\end{proposition}
\begin{proof}
The congruences modulo \(24\) in \eqref{eq:eta-criterion} are equivalent
to \(24\mid e\Lambda_T\) by \eqref{eq:walsh-sums}. The parity conditions
hold automatically for \(k\ge2\); for \(M=p\) they require \(2\mid e\).
Thus \(e_T=24/\gcd(24,\Lambda_T)\) for \(k\ge2\), and
\(e_{\{p\}}=\operatorname{lcm}(2,24/\gcd(24,p-1))\) for \(M=p\).
Equation \eqref{eq:cusp-orders} gives the divisor. In the Euler product
only \(1-q\), from \(\eta(\tau)\), contributes to the linear coefficient
after the leading power of \(q\) is removed.
\end{proof}

\begin{remark}\label{rem:eigenunits}
Suppose a nonconstant modular unit \(u\) satisfies
\(u\circ W_Q\in\C^\times u^{\pm1}\) for every \(Q\mid M\).
On cusp denominators \(W_Q\) acts by \(c\mapsto c*Q\). Indeed the matrix
in the proof of Proposition~\ref{prop:al-law} sends \(1/c\) to
\((Qx+cy)/(M+cQ)\). Its common factor is \(\gcd(c,Q)\), since it
divides the determinant \(Q\) and \(y\) is a unit modulo \(Q\).
The reduced denominator has gcd \(c*Q\) with \(M\).
Hence the vectors \((\chi_T(c))_{c\mid M}\) span the distinct character
lines. The divisor of \(u\) has degree zero and lies on one such nontrivial line.
By Theorem~\ref{thm:eta-criterion} and the diagonalization of
\eqref{eq:cusp-orders}, its eta-exponent vector is a multiple of \(v_T\).
This multiple is an integer, since the coordinate at \(d=1\) is that
integer. Proposition~\ref{prop:minimal} now shows that \(u\) is a
constant times an integral power of \(u_T\).
\end{remark}

\subsection{The exact transformation law}
For \(Q\mid M\) take \(W_Q\) represented by an integral matrix
\(\left(\begin{smallmatrix}Qx&y\\Mz&Qw\end{smallmatrix}\right)\)
of determinant \(Q\). The induced involutions have multiplication
\(Q*Q'=QQ'/\gcd(Q,Q')^2\).

\begin{proposition}\label{prop:al-law}
For every \(Q\mid M\) and nonempty \(T\),
\begin{equation}\label{eq:al-law}
 u_T\circ W_Q=s_T(Q)c_{T,Q}u_T^{\chi_T(Q)},\qquad
 c_{T,Q}=\begin{cases}
 p^{e_T2^{k-2}},&T=\{p\},\ p\mid Q,\\
 1,&\text{otherwise}.
 \end{cases}
\end{equation}
Here \(s_T\) is a character with values in \(\{\pm1\}\) and \(s_T(M)=1\).
It is trivial for \(k=1\), for even \(M\), and for \(k\ge3\). If
\(M=pq\) is odd, then for \(\ell\in T\) and \(\ell'=M/\ell\),
\[
 s_T(p)=s_T(q)=\left(\frac{\ell}{\ell'}\right)^{e_T}.
\]
\end{proposition}
\begin{proof}
Set \(M'=M/Q\) and choose \(Qx-M'y=1\), so one may take
\(W_Q=\left(\begin{smallmatrix}Qx&y\\ M&Q\end{smallmatrix}\right)\).
Write \(d=d_1d_2\), with \(d_1\mid Q,\ d_2\mid M'\), and put
\(d^*=(Q/d_1)d_2\). Then \(dW_Q\tau=\gamma_d(d^*\tau)\), where
\[
 \gamma_d=\begin{pmatrix}d_1x&d_2y\\M'/d_2&Q/d_1\end{pmatrix}
 \in\mathrm{SL}_2(\Z).
\]
The eta transformation law gives
\(\eta(dW_Q\tau)=\nu(\gamma_d)d_1^{-1/2}(M\tau+Q)^{1/2}
\eta(d^*\tau)\). Since \(\chi_T(d^*)=\chi_T(Q)\chi_T(d)\) and the
weights sum to zero, the multiplier of \(R_T^{\chi_T(Q)}\) is
\[
 E\prod_{d\mid M}d_1^{-\chi_T(d)/2},
 \qquad E=\prod_{d\mid M}\nu(\gamma_d)^{\chi_T(d)}.
\]
The positive factor gives \(c_{T,Q}\) after taking the \(e_T\)-th power
by \eqref{eq:walsh-sums}. For \(Q=M\) the representative
\(\left(\begin{smallmatrix}0&-1\\M&0\end{smallmatrix}\right)\) and
\(\eta(-1/z)=(-iz)^{1/2}\eta(z)\) give phase \(1\), which also proves
the assertion at prime level.

For the moment let \(s_T(Q)\) denote the phase, before it is known to
be a sign. Composition gives \(s_T(Q*Q')=s_T(Q)s_T(Q')^{\chi_T(Q)}\).
In particular \(s_T(M/Q)=s_T(Q)\), since \(s_T(M)=1\). We may therefore
choose the divisor \(Q\) for which the complementary divisor is odd
when \(M\) is even. Suppose now that
\(M'\) is odd. With \(c>0\) odd, the eta multiplier is
\cite[Ch.~4, Thm.~2]{knopp}
\[
 \nu\begin{pmatrix}a&b\\c&d\end{pmatrix}
 =\left(\frac{d}{c}\right)
 \zeta_{24}^{(a+d)c-bd(c^2-1)-3c},
 \qquad \zeta_{24}=e^{2\pi i/24}.
\]
Write \(\chi=\chi_T\). Summing \(ac,dc,bd(c^2-1)\) separately
over \(d_1,d_2\) with weights \(\chi(d_1d_2)\) gives, in order,
\[
 x\chi(Q)\Lambda_T,\qquad\Lambda_T,\qquad y(M'-\chi(M'))\Lambda_T .
\]
Indeed, \(\sum_{b\mid n}\chi(b)n/b=\prod_{p\mid n}(p+\chi(p))\)
and \(\sum_{b\mid n}\chi(b)b=\chi(n)\prod_{p\mid n}(p+\chi(p))\).
Each is divisible by 24 after multiplication by \(e_T\).
The term \(-3c\) contributes
\[
 -3\sum_{d_1\mid Q}\chi(d_1)\prod_{p\mid M'}(p+\chi(p)).
\]
The sum is zero if a prime in \(T\) divides \(Q\), and otherwise
equals \(2^{\omega(Q)}\), where \(\omega(Q)\) counts the prime factors.
Each prime of \(M'\) is odd, so the expression is divisible by 24
for \(k\ge3\). The product of Jacobi symbols in \(E\) is
\(\prod_{a\mid Q,\ b\mid M'}(\frac{a}{b})
=\prod_{p\mid Q,\ p'\mid M'}(\frac{p}{p'})^{2^{k-2}}\),
where the indices in the last product are primes. This too is \(1\)
for \(k\ge3\).

For \(k=2\), write \(Q=q,\ M'=p\) with \(p\) odd.
The term from \(-3c\) vanishes unless \(T=\{p\}\), when it gives
\(\zeta_{24}^{-6(p-1)}=(\frac{-1}{p})\). The single Jacobi symbol is
\((\frac qp)\). Thus
\[
 E^{e_T}=\left(\frac{q}{p}\right)^{e_T}
          \left(\frac{-1}{p}\right)^{e_T[T=\{p\}]}.
\]
This gives a sign for each \(Q\), using the equality for complementary
divisors. The composition law now makes \(s_T\) a character.
For odd \(M=pq\), take \(Q=\ell\in T\), so the second symbol drops out.
When both choices of \(\ell\) occur, quadratic reciprocity gives the
same result: if \(p,q\equiv3\pmod4\), then \(e_T\) is even.
For \(M=2p\), take \(Q=2\). When \(e_T\) is odd, divisibility
\(8\mid\Lambda_T\) forces \(p\equiv7\pmod8\) for \(T=\{2\}\)
and \(p\equiv1\pmod8\) for the other two choices. The displayed
symbol product is \(1\) in each case.
\end{proof}
At \(M=15\), \(e_{\{3,5\}}=3\) and \((\frac35)=-1\).
Thus \(W_3\) takes \(u_{\{3,5\}}\) to \(-u_{\{3,5\}}^{-1}\), giving
the invariant difference \(u_{\{3,5\}}^{-1}-u_{\{3,5\}}\).

\section{Integral structure}\label{sec:lattice}
Put \(L_M=\bigoplus_{T\ne\varnothing}\Z e_Tv_T\).
Theorem~\ref{thm:eta-criterion} identifies \(Q_M\) with \(K_M/L_M\).

\subsection{Integral coordinates}
For \(S\subseteq P\), write \(p_S=\prod_{p\in S}p\) and
\[
 x_S(r)=\sum_{\substack{d\mid M\\p_S\mid d}}r_d.
\]
Inclusion--exclusion inverts this change of integral coordinates, and
\(I_M\) is the hyperplane \(x_\varnothing=0\). Subscripts such as
\(p,q\) denote the subset \(\{p,q\}\).

\begin{lemma}\label{lem:coords}
Put \(a_S=\prod_{p\in S}(p-1)\). Then
\[
 A(r):=\sum_{d\mid M}d r_d=\sum_S a_Sx_S,\qquad
 B(r):=\sum_{d\mid M}(M/d)r_d
 =\sum_S(-1)^{|S|}(M/p_S)a_Sx_S.
\]
Also \(x_S(v_T)=(-1)^{|T|}2^{k-|S|}\) if \(T\subseteq S\), and it
is zero otherwise. Consequently \(\bigoplus_{T\ne\varnothing}\Z v_T\)
is the sublattice of \(I_M\) cut out by \(2^{k-|S|}\mid x_S\) for
every \(S\ne\varnothing\); its index is \(2^{k2^{k-1}-k}\).
\end{lemma}
\begin{proof}
Expand \(d=\prod_{p\mid d}(1+(p-1))\) and
\(M/d=M\prod_{p\mid d}(1-(p-1)/p)\).
For the second assertion, the primes outside \(S\) give the factors
\(1+\chi_T(p)\). The matrix of the nonempty \(v_T\) in the \(x_S\)
coordinates is the diagonal matrix with entries \(2^{k-|S|}\)
times the subset inclusion matrix with signed columns. Ordering
subsets by size makes the latter triangular with diagonal entries
\(\pm1\). The index follows by multiplying the diagonal factors.
\end{proof}

\subsection{The local congruences}
Write \(n_p=(p-1)/2\) for odd \(p\mid M\), and
\(\mathcal B=\{p\mid M:\ p\equiv3\pmod4\}\).
Thus \(\kappa_4=1\) exactly when \(\mathcal B\ne\varnothing\). We write \([\,\cdot\,]\) for the indicator of a condition.

\begin{lemma}\label{lem:local}
The conditions in \eqref{eq:eta-criterion} at 2 are equivalent to
\(x_p=2z_p\) for every prime \(p\mid M\), together with
\begin{equation}\label{eq:lattice-xi}
 \xi:=\sum_{p\in\mathcal B}z_p+
             \sum_{\{p,q\}\subseteq\mathcal B}x_{p,q}\equiv0\pmod2
\end{equation}
and, for even \(M\),
\begin{equation}\label{eq:lattice-even}
 z_2\equiv \sum_{p>2}p n_p x_{2,p}
                 -2\sum_{2<p<q}n_pn_qx_{2,p,q}\pmod4 .
\end{equation}
Here and below the prime indices divide \(M\).
These conditions have index \(2^{k+2\kappa_2+\kappa_4}\) in \(I_M\).
At 3, write \(r=\sum_{T\ne\varnothing}t_Tv_T\).
The conditions in \eqref{eq:eta-criterion} are equivalent to
\(t_T\in3\Z_3\) for each \(T\) with \(3\nmid\Lambda_T\).
\end{lemma}
\begin{proof}
The parity conditions give \(x_p=2z_p\). For odd \(M\), the only terms
in Lemma~\ref{lem:coords} that can survive modulo 8 have
\(1\le |S|\le2\). Their contributions are already divisible by 4,
so the additional odd signed factors in \(B(r)\) do not affect them.
Both \(A(r)\) and \(B(r)\) are congruent to \(4\xi\pmod8\).

Now write \(M=2m\) with \(m\) odd. In \(B(r)\) the terms indexed by
\(S\) avoiding 2 are zero modulo 8. Since \(m^2\equiv1\pmod8\), the
remaining terms give, after division by 2,
\begin{align*}
 m B(r)/2&\equiv -z_2+\sum_{p>2}p n_p x_{2,p}
               -2\sum_{2<p<q} n_pn_q x_{2,p,q}\pmod4,\\
 A(r)/2&\equiv z_2+\sum_{p>2} n_p x_{2,p}
               +2\sum_{2<p<q} n_pn_q x_{2,p,q}+2\xi\pmod4 .
\end{align*}
Substitution of the first congruence into the second leaves \(2\xi\)
because \((p+1)n_p=(p^2-1)/2\equiv0\pmod4\). The parities have
index \(2^k\). If present, \eqref{eq:lattice-even} specifies \(z_2\)
modulo 4, and \eqref{eq:lattice-xi} specifies one \(z_p\) modulo 2 when
\(\mathcal B\) is nonempty. Those restrictions are independent.

For the assertion at 3 the \(t_T\) belong to \(\Z_3\) by
Lemma~\ref{lem:coords}. The two congruences are the vanishing modulo
3 of \(\sum_T(-1)^{|T|}\Lambda_Tt_T\) and \(\sum_T\Lambda_Tt_T\).
To have \(3\nmid\Lambda_T\), \(T\) must contain all prime divisors
congruent to 2 modulo 3 and exclude all those congruent to 1 modulo 3.
Only membership of 3 can vary. Thus there are at most two eligible
nonempty \(T\), and if there are two they have opposite parity.
Addition and subtraction separate their two coefficients modulo 3.
\end{proof}

\subsection{The structure of the quotient}
\begin{theorem}\label{thm:lattice}
The quotient \(Q_M\) is a finite \(2\)-group, with the invariant
factors of Theorem~\ref{intro-integral}. For \(k\ge2\), if
\(\alpha_T=\max(0,3-v_2(\Lambda_T))\), then
\begin{equation}\label{eq:lattice-order}
 |Q_M|=2^{\,k2^{k-1}-2k-2\kappa_2-\kappa_4+\sum_{T\ne\varnothing}\alpha_T}.
\end{equation}
\end{theorem}
\begin{proof}
For \(\ell\ge5\), all three lattices \(L_M,K_M,I_M\) agree over
\(\Z_\ell\) by Lemma~\ref{lem:coords} and \(e_T\mid24\).
At 3, Lemma~\ref{lem:local} gives \(K_M\otimes\Z_3=L_M\otimes\Z_3\).
The 2-parts of \([I_M:L_M]\) and \([I_M:K_M]\) from those lemmas
give \eqref{eq:lattice-order}. For \(k=1\), Proposition~\ref{prop:minimal}
already gives \(L_M=K_M\).

Work henceforth over \(\Z_2\). Suppose first that \(M\) is odd with
\(k\ge3\), or even with \(k\ge4\). Each \(\Lambda_T\) is divisible by
8, so \(e_T\) is odd for all \(T\).
Lemma~\ref{lem:coords} gives \(I_M/L_M=\bigoplus_{S\ne\varnothing}
\Z/2^{\,k-|S|}\Z\) in the \(x_S\) coordinates. Imposing the singleton
parities replaces those factors by \(z_p\in\Z/2^{k-2}\Z\).
The pairs also give factors of order \(2^{k-2}\).
When \(\mathcal B\ne\varnothing\), replace \(z_{p_0}\) for
\(p_0\in\mathcal B\) by the full sum defining \(\xi\), taken modulo
\(2^{k-2}\). All its summands have that same modulus.
This invertible change shortens one factor of order \(2^{k-2}\)
to order \(2^{k-3}\). For even \(M\), replace \(z_2\) by
\[
 z_2-\sum_{p>2}p n_p x_{2,p}
        +2\sum_{2<p<q}n_pn_qx_{2,p,q}.
\]
The coefficient 2 on each triple makes this a well-defined change
from coordinates of order \(2^{k-3}\) into a coordinate of order
\(2^{k-2}\). It is invertible and does not enter \(\xi\). Condition
\eqref{eq:lattice-even} thus shortens a second factor of order
\(2^{k-2}\) to order \(2^{k-4}\). Counting subsets of each size
gives the multiplicities of Theorem~\ref{intro-integral}.

For \(M=pq\) odd, \(\alpha_T=1\) precisely when both factors of
\(\Lambda_T\) are \(2\bmod4\). This occurs for \(T=\mathcal B\)
if \(\mathcal B\ne\varnothing\), and never otherwise.
Equation \eqref{eq:lattice-order} gives \(Q_M=0\).
For \(M=2p\), the triples \((\alpha_{\{2\}},\alpha_{\{p\}},
\alpha_{\{2,p\}})\) for \(p\equiv1,3,5,7\bmod8\) are respectively
\((2,0,0),(1,2,2),(2,1,1),(0,2,2)\). The orders are \(1,4,4,2\).
If \(t=x_{2,p}\), the expansion of \(r\in K_M\) is
\[
 r=-z_2v_{\{2\}}-z_pv_{\{p\}}+(t-z_2-z_p)v_{\{2,p\}},\qquad
 z_2\equiv p n_p t\pmod4 .
\]
For \(p\equiv5\bmod8\), \(z_2\) is even; for \(p\equiv3\bmod8\),
\(z_p\) and \(t-z_2-z_p\) are even, by \eqref{eq:lattice-xi}
and \eqref{eq:lattice-even}. Thus \(2r\in L_M\otimes\Z_2\)
in both order-four cases, giving \((\Z/2\Z)^2\).

Finally suppose \(M=2pq\). Now \(\alpha_T=1\) for
\(T=\mathcal B,\mathcal B\cup\{2\}\) (omitting the empty set), and is
zero otherwise. Thus the order is \(2^5\).
Fourier inversion on the divisor group gives \(2t_T\in\Z_2\) and
\[
 2t_T\equiv\sum_{i\in T}z_i+\sum_{\{i,j\}\subseteq T}x_{i,j}\pmod2.
\]
Indeed, expand \(\chi_T(d)=\prod_{i\in T}(1-2[i\mid d])\) in
\(t_T=2^{-3}\sum_d\chi_T(d)r_d\), using \(x_\varnothing=0\).
For \(T=\mathcal B\) the right side is \(\xi\); for
\(T=\mathcal B\cup\{2\}\) it is
\(\xi+z_2+\sum_{i\in\mathcal B}x_{2,i}\).
Both vanish by \eqref{eq:lattice-xi}--\eqref{eq:lattice-even}.
This gives \(2r\in L_M\otimes\Z_2\) and hence \(Q_M\cong(\Z/2\Z)^5\).
\end{proof}

\begin{corollary}\label{cor:generation}
The units \(u_T\) generate all modular units modulo constants exactly for
\(k=1\), for odd \(M\) with \(k=2\), and for \(M=2p\) with
\(p\equiv1\bmod8\). For \(k\ge3\) the least \(n>0\) such that
every modular unit \(f\) satisfies \(f^n\in\C^\times\prod_Tu_T^\Z\)
is \(2^{k-2}\).
\end{corollary}

\begin{corollary}\label{cor:cuspidal-order}
Let \(k\ge2\), and let \(C(M)\) be the group of degree-zero cuspidal
divisors modulo principal cuspidal divisors. The class of
\(\sum_{c\mid M}\chi_T(c)(1/c)\) has order \(m_T\), and
\[
 |C(M)|\,|Q_M|=2^{\,k(2^{k-1}-1)}\prod_{T\ne\varnothing}m_T.
\]
\end{corollary}
\begin{proof}
If \(n(\chi_T(c))_{c\mid M}\) is principal, the invertible cusp-order
matrix in \eqref{eq:cusp-orders} forces its eta exponents to be
\((24n/\Lambda_T)v_T\). The coordinate at \(1\) is integral and must
be a multiple of \(e_T\) by Proposition~\ref{prop:minimal}. Thus
\(m_T\mid n\), and \(u_T\) attains \(n=m_T\). The divisors of the \(u_T\)
generate a lattice of index \(2^{k2^{k-1}-k}\prod_Tm_T\) in all
degree-zero cusp divisors by Lemma~\ref{lem:coords}, and of index
\(|Q_M|\) in principal cusp divisors. The index formula follows.
\end{proof}

\noindent\emph{Generators of \(Q_6\).}
Here \(u_{\{2\}}=R_{\{2\}}^6,\ u_{\{3\}}=R_{\{3\}}^4,\)
and \(u_{\{2,3\}}=R_{\{2,3\}}^{12}\). The eta quotient
\[
 F(\tau)=\frac{\eta(\tau)^5\eta(3\tau)}
                    {\eta(2\tau)\eta(6\tau)^5}
        =R_{\{2\}}^3 R_{\{3\}}^2
\]
has exponent vector \((5,-1,1,-5)\) in divisor order \(1,2,3,6\).
Its two sums in \eqref{eq:eta-criterion} are \(-24,24\), and its two
parity sums are \(-6,-4\); thus \(F\) is a modular unit.
The rational independence of the character vectors shows that \(F\)
represents a nonzero element of \(Q_6\), while
\(F^2=u_{\{2\}}u_{\{3\}}\). A second such class is represented by
\(G=R_{\{2\}}^3R_{\{2,3\}}^6\). Its eta exponents are
\((9,-9,-3,3)\), which also satisfy \eqref{eq:eta-criterion}, and
\(G^2=u_{\{2\}}u_{\{2,3\}}\). The different half-integral coordinates
at \(\{3\}\) and \(\{2,3\}\) make these classes independent, accounting
for \(Q_6\cong(\Z/2\Z)^2\).

At \(M=30\), the seven orders \(m_T\), in the order
\(\{2\},\{3\},\{5\},\{2,3\},\{2,5\},\{3,5\},\{2,3,5\}\),
are \(1,3,2,1,2,1,1\). Thus \(|Q_{30}|=32\) and
\(|C(30)|=2^9\cdot12/32=192\). For \(M=14,15,21,22,26,33,35,39\),
the products \(\prod_Tm_T\) are \(3,2,2,25,21,25,12,14\).
Dividing four times each product by \(|Q_M|\) gives the respective
orders \(|C(M)|=6,8,8,25,21,100,48,56\).

\section{Fourier expansion of Heegner values}
\label{sec:heegner}
\subsection{Heegner forms}
Let $D<0$ be the discriminant of an order of conductor $f$. We write $\Cl(D)$ for the group of proper classes of primitive positive definite binary quadratic forms of discriminant $D$, and $h(D)=|\Cl(D)|$. Call $D$ \emph{admissible} for $M$ if $(M,f)=1$ and $D$ is a square modulo $4M$. Equivalently, each prime $p\mid M$ is split or ramified in $\Q(\sqrt D)$ and is prime to $f$. Fix an integer $\beta$ with $\beta^2\equiv D\pmod{4M}$. A \emph{Heegner form} is a primitive form $Q=[A,B,C]$ of discriminant $D$ with $M\mid A$ and $B\equiv\beta\pmod{2M}$. Put
\[
 \tau_Q=\frac{-B+\sqrt D}{2A},\qquad
 Q_d=[A/d,B,dC],\qquad
 P_d=\bigl[d,\beta,(\beta^2-D)/(4d)\bigr]\ \text{in }\Cl(D)
 \quad(d\mid M),
\]
where $\sqrt D$ has positive imaginary part. The orientation depends only on $\beta\bmod 2M$. The form with $A=M$ and $B=\beta$ gives the point $\tau_\beta=(-\beta+\sqrt D)/(2M)$.
\begin{lemma}\label{lem:hhforms}
For $d\mid M$, $Q_d$ is primitive, has discriminant $D$, and has root $d\tau_Q$. If $dd'\mid M$, then $P_{dd'}=P_dP_{d'}$. Moreover $[Q_d]=[Q]P_d^{-1}$, and the class of the form defining $\tau_\beta$ is $P_M$. Every class in $\Cl(D)$ contains a Heegner form.
\end{lemma}
\begin{proof}
If a prime $\ell$ divides $A/d,B,dC$, then primitivity of $Q$ implies $\ell\nmid C$, so $\ell\mid d$ and $\ell^2\mid A$. For odd $\ell$ this forces $\ell^2\mid D$, hence $\ell\mid f$. For $\ell=2$ it gives $D/4\equiv(B/2)^2\pmod4$, whereas $2\mid D$ and odd $f$ would give $D/4\equiv2,3\pmod4$ by the classification of fundamental discriminants. Both contradict $(f,M)=1$. The same discriminant argument proves primitivity of the form defining $P_d$ and shows $\gcd(A/d,d,B)=1$. The discriminant and the root follow by substitution.

The composition of $[a,b,a'c]$ and $[a',b,ac]$ with $\gcd(a,a',b)=1$ is $[aa',b,c]$ \cite[\S3]{cox}. Apply this first to $d,d'$ and then to $A/d,d$ with middle coefficient $B$. Since $B\equiv\beta\pmod{2d}$, the form $[d,B,(B^2-D)/(4d)]$ represents $P_d$. The form defining $\tau_\beta$ represents $P_M$ by definition.

Given $X$, choose $[a,b,c]\in XP_M^{-1}$ with $(a,M)=1$. Indeed, modulo each $p\mid M$, primitivity supplies a nonzero column on which the form is nonzero. Complete these columns to determinant-one matrices and lift by $\mathrm{SL}_2(\Z)\twoheadrightarrow\mathrm{SL}_2(\Z/M\Z)$. Choose $B\equiv b\pmod{2a}$ and $B\equiv\beta\pmod{2M}$, which is possible since the two residues have the same parity. Set $C'=(B^2-D)/(4aM)$. The form $[a,B,MC']$ is equivalent to $[a,b,c]$, hence primitive. A prime common to the coefficients of $[aM,B,C']$ must therefore divide $M$. In that case the preceding discriminant argument forces it to divide $f$. Composition now gives $[aM,B,C']\in X$.
\end{proof}
The $2^k$ classes $[Q]P_d^{-1}$ form a translated family indexed by divisors, possibly with repetitions. The orientation $\beta$ selects the prime ideal $\mathfrak p_p=(p,(-\beta+\sqrt D)/2)$ representing $P_p$.

\subsection{The character formula}
The function $\lambda(\tau)=\log(\sqrt{\operatorname{Im}\tau}\,|\eta(\tau)|^2)$ is invariant under $\mathrm{SL}_2(\Z)$. Write $\lambda(X)$ for its value at the root of a form in the class $X$. For a character $\psi$ of $\Cl(D)$ let
\[
 S(\psi)=\sum_{X\in\Cl(D)}\psi(X)\lambda(X),\qquad
 E_T(\psi)=\prod_{p\in T}(1-\psi(P_p))\prod_{p\in P\setminus T}(1+\psi(P_p)).
\]
Put $\delta_{\{p\}}=2^{k-3}\log p$ and $\delta_T=0$ for $|T|\ge2$.
\begin{theorem}[Fourier formula for Heegner values]\label{thm:hhconv}
For every admissible $D$, every Heegner form $Q$, and every nonempty $T$,
\begin{equation}\label{eq:hhpoint}
 \log |R_T(\tau_Q)|=\tfrac12\sum_{d\mid M}\chi_T(d)\lambda([Q]P_d^{-1})+\delta_T .
\end{equation}
Thus $|R_T(\tau_Q)|$ is well defined on the class $X=[Q]$ with its fixed orientation, and we write $|R_T(\tau_X)|$. For every class character $\psi$,
\begin{equation}\label{eq:hhtwist}
 \sum_{X\in\Cl(D)}\psi(X)\log |R_T(\tau_X)|
   =\tfrac12 E_T(\psi)S(\psi)+\mathbf1_{\psi=1}h(D)\delta_T .
\end{equation}
\end{theorem}
\begin{proof}
Lemma~\ref{lem:hhforms} gives
$\lambda([Q]P_d^{-1})=\tfrac12\log(d\operatorname{Im}\tau_Q)+2\log|\eta(d\tau_Q)|$.
The sum of the $\chi_T(d)$ is zero. Their weighted sum against $\log d$ equals $-2^{k-1}\log p$ for $T=\{p\}$ and zero otherwise. This proves \eqref{eq:hhpoint}. Multiply by $\psi(X)$, sum over $X$, and put $Y=XP_d^{-1}$. The inner sum becomes $\psi(P_d)S(\psi)$. Finally $\sum_{d\mid M}\chi_T(d)\psi(P_d)=E_T(\psi)$ by multiplication of the local factors.
\end{proof}
A product raised to complex character weights is interpreted by exponentiating the sum of real logarithms in \eqref{eq:hhtwist}. The subgroup $H=\langle P_p:p\mid M\rangle$ describes which class-group data the full family retains. If $\psi|_H=1$, all $E_T(\psi)$ vanish for nonempty $T$. Conversely, if $\psi|_H\ne1$, put every $p$ with $\psi(P_p)=-1$ in $T$ and every $p$ with $\psi(P_p)=1$ outside $T$. If $T$ is still empty, include one $p$ for which $\psi(P_p)$ is nonreal. Then $E_T(\psi)\ne0$. Fourier inversion therefore recovers $\lambda$ from all the normalized logarithms up to an additive constant on each coset of $H$.

\begin{corollary}\label{cor:hhpair}
\begin{enumerate}[label=(\alph*)]
\item The product $\prod_X |R_T(\tau_X)|$ equals $e^{h(D)\delta_T}$.
\item If $P_p=1$ for some $p\in T$, then $|R_T(\tau_X)|=e^{\delta_T}$ for every $X$.
\item If $P_p^2=1$, then
$\log|R_T(\tau_{XP_p})|-\delta_T
 =\chi_T(p)(\log|R_T(\tau_X)|-\delta_T)$.
\item If $P_d^2=1$ for every $d\mid M$ and $P_M=1$, then $|R_T(\tau_X)|=e^{\delta_T}$ for every $X$ and every odd $|T|$.
\end{enumerate}
\end{corollary}
\begin{proof}
For the first assertion take $\psi=1$ in \eqref{eq:hhtwist}. For the next two, pair divisors in \eqref{eq:hhpoint} by toggling membership of $p$. For the last one, pair $d$ with $M/d$ and use $P_{M/d}=P_d$ and $\chi_T(M/d)=-\chi_T(d)$.
\end{proof}
The third assertion applies to a ramified prime, and the fourth applies at the Fricke fixed points considered below. For the second, take $M=2\cdot59$ and $D=-23$. The principal form $x^2+xy+6y^2$ represents $59$ at $(5,2)$, so $|\eta(\tau)\eta(118\tau)|=|\eta(2\tau)\eta(59\tau)|$ at all three Heegner classes.

\subsection{Arithmetic nature of the values}
Recall that $u_T=R_T^{e_T}$, where $e_T$ includes the character condition at prime level.
\begin{proposition}\label{prop:hhalg}
Let $\tau_0\in\Hh$ be imaginary quadratic.
\begin{enumerate}[label=(\alph*)]
\item $R_T(\tau_0)^{24}$ is an algebraic integer. It is a unit for $|T|\ge2$, and divides $p^{6\cdot2^k}$ in the ring of algebraic integers for $T=\{p\}$.
\item If $\tau_0$ is a Heegner point of admissible discriminant $D$, then $u_T(\tau_0)$ lies in the ring class field $L_D$ of the order of discriminant $D$.
\item If in addition $\Cl(D)$ has exponent at most two and $|T|\ge2$, then $u_T(\tau_0)/|u_T(\tau_0)|$ is a root of unity.
\end{enumerate}
\end{proposition}
\begin{proof}
Put $\Delta=\eta^{24}$. For $\gamma=\left(\begin{smallmatrix}a&b\\c&d\end{smallmatrix}\right)\in\mathrm{SL}_2(\Z)$, $e\mid M$ and $g=(c,e)$, write
$\left(\begin{smallmatrix}e&0\\0&1\end{smallmatrix}\right)\gamma
=\gamma'\left(\begin{smallmatrix}g&B\\0&e/g\end{smallmatrix}\right)$ by completing the primitive first column $(ea/g,c/g)^t$ of $\gamma'$ to determinant one. Weight twelve gives
\[
 \Delta(e\gamma\tau)=(g/e)^{12}(c\tau+d)^{12}
      \Delta\bigl((g\tau+B)/(e/g)\bigr).
\]
Hence $R_T^{24}\circ\gamma$ is $\prod_{e\mid M}((c,e)/e)^{12\chi_T(e)}$ times a Laurent series in $q^{1/M}$ with algebraic-integer coefficients and a root of unity as leading coefficient. The prefactor is one for $|T|\ge2$. For $T=\{p\}$ it is one when $p\mid c$ and $p^{6\cdot2^k}$ otherwise. The monic polynomial with these translates as roots has invariant coefficients holomorphic on $\Hh$. Translation invariance leaves integral powers of $q$. Subtract powers of $j=q^{-1}+744+\cdots\in\Z((q))$ with algebraic-integer coefficients to cancel the polar part. The remainder is holomorphic on compact $X(1)$, hence constant, and its constant term is an algebraic integer. Transitivity gives integrality over $\Z[j]$. The same argument applies to $R_T^{-24}$ for $|T|\ge2$ and to $p^{6\cdot2^k}R_T^{-24}$ for a singleton. Specialize at the algebraic integer $j(\tau_0)$.

The modular unit $u_T$ has rational Fourier coefficients, so it is defined over $\Q$ on $X_0(M)$. By complex multiplication $L_D=\Q(\sqrt D,j(\tau_0))$. Choose a model of $E=\C/(\Z\tau_0+\Z)$ over $L_D$. Its endomorphisms are defined there, since their differentials lie in $\Q(\sqrt D)$ and determine them in characteristic zero \cite{gross,shimura}. The cyclic subgroup generated by $1/M$ is the common kernel $E[\mathfrak n]$ for $\mathfrak n=(M,(-\beta+\sqrt D)/2)$. Indeed the second generator divided by $M$ belongs to $\Z\tau_0+\Z$ by the Heegner congruences, and the invertible ideal $\mathfrak n$ has norm $M$. Thus the pair defining the point of $X_0(M)$ is rational over $L_D$. For the last assertion $L_D/\Q$ is abelian, since conjugation acts by inversion on $\Cl(D)$ and every class equals its inverse. With $u=u_T(\tau_0)$, part (a) implies $u$ is an algebraic unit. Each conjugate of the unit $u/\bar u$ has the form $\sigma(u)/\overline{\sigma(u)}$ by commutativity, so Kronecker's theorem makes it a root of unity. Now $(u/|u|)^2=u/\bar u$.
\end{proof}

At \(M=17\), \(D=-4\) and \(\tau=(-4+i)/17\), the exponent in Proposition~\ref{prop:minimal} is \(6\).
Equation~\eqref{eq:hhpoint} gives \(\lvert R_{17}(\tau)^3\rvert^2=17^{3/2}\). An element of \(\Q(i)\) has rational squared absolute value, so the smaller exponent \(3=24/\gcd(24,16)\) would fail the field assertion.

\section{Genus characters and their Heegner values}
\label{sec:genus}
For the evaluations by genus characters assume that $D$ is admissible and either fundamental or of the form $4D_0$ with $D_0\equiv1\pmod4$ fundamental. In the latter case $M$ is odd. The class-number-one result below holds for every admissible order of class number one.

\subsection{Kronecker's limit formula}
The real characters of $\Cl(D)$ are the genus characters. Each nontrivial genus character corresponds to a factorization of $D$ (of $D_0$ in the conductor-two case) as $D_1D_2$ into fundamental discriminants with $D_1>1$ and $D_2<0$. Its value on $X$ is $\chi_{D_1}(n)$ for any integer $n$ prime to $2D$ represented by $X$ \cite[\S6]{cox}. Put
\[
 U(\psi)=\varepsilon_{D_1}^{\,c_\psi w(D) h(D_1)h(D_2)/w(D_2)},\qquad
 c_\psi=\begin{cases}
 1,&D\ \text{fundamental},\\
 3-\chi_{D_1}(2)-\chi_{D_2}(2),&D=4D_0.
 \end{cases}
\]
Here $h(D_1)$ is the ordinary class number, and $\varepsilon_{D_1}>1$ generates the unit group modulo $\{\pm1\}$, allowing norm $-1$.
\begin{proposition}\label{prop:hhkronecker}
For every nontrivial genus character $\psi$, $S(\psi)=-\log U(\psi)$.
\end{proposition}
For fundamental $D$, the genus-character $L$-function factors as $L(s,\chi_{D_1})L(s,\chi_{D_2})$. For a root $\tau$ in $X^{-1}$, the partial zeta function of $X$ is $w(D)^{-1}(2/\sqrt{|D|})^s\sum' (\Im\tau)^s|m\tau+n|^{-2s}$, where the sum omits $(m,n)=(0,0)$. Its constant term at $s=1$ is $C_D-4\pi\lambda(X^{-1})/(w(D)\sqrt{|D|})$ by Kronecker's first limit formula \cite[(2.8), (2.12)--(2.16)]{bcz}. Since $\psi(X^{-1})=\psi(X)$, twisting cancels the common pole and $C_D$, giving
\[ S(\psi)=-\frac{w(D)\sqrt{|D|}}{4\pi}\,L(1,\chi_{D_1})L(1,\chi_{D_2}),\qquad \chi_d(n)=\Bigl(\frac dn\Bigr). \]
For the ordinary class numbers the two formulas are $L(1,\chi_{D_1})=2h(D_1)\log\varepsilon_{D_1}/\sqrt{D_1}$ and $L(1,\chi_{D_2})=2\pi h(D_2)/(w(D_2)\sqrt{|D_2|})$ \cite{siegel,bcz}. Substitution gives $S(\psi)=-\log U(\psi)$. To pass to conductor two, extend $\lambda$ to lattices by homothety invariance, with $\lambda(\Z\tau+\Z)=\lambda(\tau)$.
\begin{lemma}\label{lem:hhtwo}
If $L_1,L_2,L_3$ are the three index-two sublattices of a complex lattice $L$, then
$\lambda(L_1)+\lambda(L_2)+\lambda(L_3)=3\lambda(L)-\tfrac12\log2$.
\end{lemma}
\begin{proof}
For $L=\Z\tau+\Z$ the three parameters are $2\tau$, $\tau/2$ and $(\tau+1)/2$. The eta product identity
$\eta(2\tau)\eta(\tau/2)\eta((\tau+1)/2)=e^{\pi i/24}\eta(\tau)^3$ follows from the product expansions (equivalently, Weber's identity \cite[\S12]{cox}). Take absolute values and include the three imaginary-part terms.
\end{proof}
\begin{proof}[Proof of Proposition~\ref{prop:hhkronecker} for $D=4D_0$]
Let $\pi:\Cl(4D_0)\longrightarrow\Cl(D_0)$ extend ideals to the maximal order. Fix a maximal-order ideal $L$ of class $Y$ and write $\mathcal O$ for its multiplier ring. An index-two sublattice $I$ contains $2L$ and is stable under $\Z+2\mathcal O$. If $\mathcal O I=I$, then $I=\mathfrak p L$ for a prime of norm two. Otherwise $\mathcal O I=L$ and $I$ is a proper ideal of the order $\Z+2\mathcal O$, hence invertible. Conversely, scale a proper ideal above $Y$ so that $\mathcal O I=L$. Then $2L\subseteq I\subseteq L$. Both inclusions are strict because $2L$ and $L$ have multiplier ring $\mathcal O$, so $[L:I]=2$. Two such sublattices in the same class differ by a unit of $\mathcal O$. For $D_0=-3$ there is no nontrivial character since $h(-12)=1$. Otherwise the units $\pm1$ fix every sublattice. Thus when $2$ is inert, all three sublattices give distinct classes above $Y$. Twisting Lemma~\ref{lem:hhtwo} by a nontrivial genus character gives $S_{4D_0}(\psi\circ\pi)=3S_{D_0}(\psi)$. When $2=\mathfrak p\bar{\mathfrak p}$ splits, the fibre consists of one class $X$ and
\[
 \lambda(X)=3\lambda(Y)-\lambda(Y[\mathfrak p])-\lambda(Y[\bar{\mathfrak p}])-\tfrac12\log2 .
\]
The twisted sum now gives the multiplier $3-\psi(\mathfrak p)-\psi(\bar{\mathfrak p})$. In the split case $\psi(\mathfrak p)=\psi(\bar{\mathfrak p})=\chi_{D_1}(2)=\chi_{D_2}(2)$ by the genus character formula at a prime ideal. In the inert case $\chi_{D_1}(2)\chi_{D_2}(2)=\chi_{D_0}(2)=-1$. The kernel of $\pi$ has order one or three, so every real character descends. These facts give $c_\psi$ in both cases.
\end{proof}

\subsection{Values of the genus characters}
The prime ideal representing $P_p$ has norm $p$, also after extension to the maximal order when the conductor is two. The prime-ideal formula for a genus character \cite[p.~2129]{bcz} uses the discriminant factor prime to $p$. Thus it applies also at a ramified prime and at $p=2$, and gives
\begin{equation}\label{eq:hhgenusprime}
 \psi(P_p)=
 \begin{cases}
  \chi_{D_1}(p),&p\nmid D_1,\\
  \chi_{D_2}(p),&p\mid D_1 .
 \end{cases}
\end{equation}
For a genus character $\psi$ define $T(\psi)=\{p\mid M:\psi(P_p)=-1\}$. For real characters \(\psi,\psi'\) one has \(T(\psi\psi')=T(\psi)\mathbin{\triangle}T(\psi')\), where \(\triangle\) denotes symmetric difference. Directly from the local factors,
\begin{equation}\label{eq:hhselect}
 E_T(\psi)=
 \begin{cases}
 2^k,&T=T(\psi),\\
 0,&T\ne T(\psi)
 \end{cases}
 \qquad(\psi\ \text{real}).
\end{equation}
\begin{theorem}[Products over a genus]\label{thm:hhselection}
For each nontrivial genus character $\psi$,
\[
 \prod_{X\in\Cl(D)}|R_T(\tau_X)|^{\psi(X)}
 =\begin{cases}
 U(\psi)^{-2^{k-1}},&T=T(\psi),\\
 1,&T\ne T(\psi).
 \end{cases}
\]
Let $g=[\Cl(D):\Cl(D)^2]$ and let $\mathfrak g$ be any genus, a coset of $\Cl(D)^2$. Then
\begin{equation}\label{eq:hhgenusnorm}
 \prod_{X\in\mathfrak g}|R_T(\tau_X)|
 =e^{\delta_T h(D)/g}
  \prod_{\substack{\psi\ne1\\T(\psi)=T}} U(\psi)^{-\psi(\mathfrak g)2^{k-1}/g}.
\end{equation}
\end{theorem}
\begin{proof}
Insert \eqref{eq:hhselect} and Proposition~\ref{prop:hhkronecker} in \eqref{eq:hhtwist}. The indicator of the coset $\mathfrak g$ is $g^{-1}\sum_\psi\psi(\mathfrak g)\psi(X)$, where the sum runs over all genus characters. This gives \eqref{eq:hhgenusnorm}.
\end{proof}
As $\psi$ ranges over genus characters, the subsets $T(\psi)$ are exactly those $T$ for which $\chi_T(d)=1$ whenever $P_d$ is a square in $\Cl(D)$. Indeed the map from the divisor group to $\Cl(D)/\Cl(D)^2$ is a homomorphism, and every character of its image extends to this vector space over $\mathbf F_2$.
\begin{remark}\label{def:hhregular}
If $\Cl(D)$ has exponent at most two, all its characters are real. More generally, suppose the $P_p$ have order at most two, as happens when every $p\mid M$ ramifies, and put $H=\langle P_p:p\mid M\rangle$. By the factors in $E_T$, a character with $E_T(\psi)\ne0$ must restrict to $P_d\mapsto\chi_T(d)$ on $H$. If this assignment violates a relation in $H$, all these coefficients vanish and $|R_T(\tau_X)|=e^{\delta_T}$. Otherwise every extension of this character to $\Cl(D)$ is real exactly when
\begin{equation}\label{eq:hhregular}
 \Cl(D)^2\subseteq H,\qquad
 \chi_T(d)=1\ \text{whenever }P_d\in\Cl(D)^2 .
\end{equation}
Indeed, extensions differ by characters of $\Cl(D)/H$, all of which are real exactly when $\Cl(D)^2\subseteq H$. Under this inclusion every extension is real precisely when its prescribed restriction is trivial on $\Cl(D)^2$.
\end{remark}
\begin{theorem}[Individual absolute values]\label{thm:hhclosed}
Suppose $E_T(\psi)=0$ for every nonreal character $\psi$. For $X=[Q]$,
\begin{equation}\label{eq:hhclosed}
 |R_T(\tau_X)|=e^{\delta_T}
 \prod_{\substack{\psi\ne1\\T(\psi)=T}}
 U(\psi)^{-\psi(X)2^{k-1}/h(D)} ,
\end{equation}
the product being over genus characters with $T(\psi)=T$. In particular, an empty product gives $|R_T(\tau_X)|=e^{\delta_T}$.
\end{theorem}
\begin{proof}
Fourier inversion applied to \eqref{eq:hhtwist} gives
$\log |R_T(\tau_X)|-\delta_T=(2h(D))^{-1}\sum_\psi\overline{\psi(X)}E_T(\psi)S(\psi)$.
By hypothesis only real characters contribute. Apply \eqref{eq:hhselect} and Proposition~\ref{prop:hhkronecker}.
\end{proof}
When $\Cl(D)$ has exponent at most two and $|T|\ge2$, Proposition~\ref{prop:hhalg} shows that the phase of $R_T$ is a root of unity because its $e_T$-th power is the phase of $u_T$. Determining the phase itself requires the chosen representative.

\begin{corollary}[Class number two]\label{cor:hhclass2}
Let $D$ be fundamental with $h(D)=2$, write $\psi$ for the nontrivial character, and put $U_D=U(\psi)^{-1}$. The set $T(\psi)$ consists of primes $p\mid M$ which the principal form does not represent. For $T\ne T(\psi)$ one has $|R_T(\tau_X)|=e^{\delta_T}$. For $T=T(\psi)$ it equals $e^{\delta_T}U_D^{\psi(X)2^{k-2}}$. In particular, for $M=pq$ with both primes nonprincipal,
\[
 |R_{\{p,q\}}(\tau_\beta)|=U_D .
\]
\end{corollary}
\begin{proof}
The class $P_p$ is nontrivial exactly when $\psi(P_p)=-1$. Apply \eqref{eq:hhclosed} with $h(D)=2$. For the last assertion $P_M=P_pP_q=1$.
\end{proof}
The class-number-two classification \cite{baker,watkins} gives the following complete list of fundamental discriminants. In each column $D_1$ is the unit subscript and $D_2=D/D_1$. Proposition~\ref{prop:hhkronecker} gives the exponent using reduced forms as in the appendix. For $-235=5(-47)$, the reduced forms of discriminant $-47$ are $[1,1,12],[2,\pm1,6],[3,\pm1,4]$. Since $h(5)=1$ (the sole reduced pair is $(1,1)$), the exponent is $-5$.
\begin{center}\small
\begin{tabular}{c|ccccccccc}
$D$&$-15$&$-20$&$-24$&$-35$&$-40$&$-51$&$-52$&$-88$&$-91$\\
$U_D$&$\varepsilon_5^{-1/3}$&$\varepsilon_5^{-1/2}$&$\varepsilon_8^{-1/3}$&
$\varepsilon_5^{-1}$&$\varepsilon_5^{-1}$&$\varepsilon_{17}^{-1/3}$&
$\varepsilon_{13}^{-1/2}$&$\varepsilon_8^{-1}$&$\varepsilon_{13}^{-1}$\\[3pt]
$D$&$-115$&$-123$&$-148$&$-187$&$-232$&$-235$&$-267$&$-403$&$-427$\\
$U_D$&$\varepsilon_5^{-3}$&$\varepsilon_{41}^{-1/3}$&$\varepsilon_{37}^{-1/2}$&
$\varepsilon_{17}^{-1}$&$\varepsilon_{29}^{-1}$&$\varepsilon_5^{-5}$&
$\varepsilon_{89}^{-1/3}$&$\varepsilon_{13}^{-3}$&$\varepsilon_{61}^{-1}$
\end{tabular}
\end{center}
For example, take $D=-24$ and distinct primes $p,q$ represented by $2x^2+3y^2$. Then $|R_{\{p,q\}}(\tau_\beta)|=(1+\sqrt2)^{-1/3}$. The Fricke values at $M=6,10,22,58$ correspond respectively to $D=-24,-40,-88,-232$, with both primes ramified.

\subsection{Class number one}
\begin{theorem}\label{thm:hhclass1}
Suppose $h(D)=1$ and $D$ is admissible, with arbitrary conductor. At any oriented Heegner point $\tau$,
\[
 R_T(\tau)^{24}=1\quad (|T|\ge2),\qquad
 R_{\{p\}}(\tau)^{24}=\pi_p^{6\cdot2^k},
\]
where $\pi_p$ generates the ideal $(p,(\beta+\sqrt D)/2)$ of the order $\mathcal O_D$. With the least modular exponent $e_T$, $u_T(\tau)\in K=\Q(\sqrt D)$. For $|T|\ge2$ this is a root of unity in $K$, and $u_{\{p\}}(\tau)$ is a root of unity in $K$ times $\pi_p^{e_{\{p\}}2^{k-2}}$. At prime level $e_{\{p\}}$ is even, so this last exponent is integral.
\end{theorem}
\begin{proof}
Let $\mathfrak a=\Z\tau+\Z$ and $\mathfrak n_d=(d,(-\beta+\sqrt D)/2)$. This is an invertible ideal of norm $d$, with $\mathfrak n_d\bar{\mathfrak n}_d=(d)$. The lattice $\Z d\tau+\Z$ equals $d\mathfrak a\mathfrak n_d^{-1}$. Indeed, $\bar{\mathfrak n}_d\mathfrak a$ is contained in $\Z d\tau+\Z$, by the quadratic equation for $\tau$ and $B\equiv\beta\bmod 2d$, and both have index $d$ in $\mathfrak a$. Write $\mathfrak a=\alpha\mathcal O_D$ and $\mathfrak n_d=\nu_d\mathcal O_D$ with $\nu_d=\prod_{p\mid d}\nu_p$. As a function of lattices, $\Delta$ is homogeneous of degree $-12$. Its fixed normalization cancels in the product, giving
\[
 R_T(\tau)^{24}
 =\prod_{d\mid M}(\nu_d/d)^{12\chi_T(d)}
 =\prod_{p\mid M}\bar\nu_p^{-12\sum_{d\mid M,\ p\mid d}\chi_T(d)}.
\]
Here $N(\nu_d)=d$. The inner sum is $-2^{k-1}$ for $T=\{p\}$ and zero otherwise. Taking $\pi_p=\bar\nu_p$ gives the two formulas, since every unit of an imaginary quadratic order has order dividing twelve. The modular value belongs to $K$ by Proposition~\ref{prop:hhalg} because $L_D=K$. Its quotient by the stated power of $\pi_p$ in the singleton case, or the value itself in the other case, has 24th power one.
\end{proof}
The class-number-one classification and the order class number formula \cite{watkins,cox} give the thirteen discriminants $-3$, $-4$, $-7$, $-8$, $-11$, $-12$, $-16$, $-19$, $-27$, $-28$, $-43$, $-67$ and $-163$. The proof works because all $2^k$ lattices are homothetic. The theorem locates $u_T(\tau)$ in $K$. The value $R_T(\tau)$ itself can lie outside $K$.

\begin{example}[Exact values and their phases]\label{ex:hhphases}
In each case $\tau=(-\beta+\sqrt D)/(2M)$ and square roots of negative integers have positive imaginary part. Subscripts abbreviate the corresponding sets of primes.
\begin{enumerate}[label=(\roman*)]
\item $D=-4$, $M=10$, $\beta=6$: $R_2=1+i,\ R_5=1+2i,\ R_{2,5}=1$.
\item $D=-7$, $M=22$, $\beta=9$: $R_2=(1+\sqrt{-7})/2,\ R_{11}=-2+\sqrt{-7},\ R_{2,11}=1$.
\item $D=-3$, $M=91$, $\beta=19$: $R_7=-(1+3\sqrt{-3})/2,\ R_{13}=-(5+3\sqrt{-3})/2,\ R_{7,13}=1$.
\item $D=-8$, $M=66$, $\beta=16$: $R_2=(\sqrt{-2})^2,\ R_3=(1-\sqrt{-2})^2,\ R_{11}=(3-\sqrt{-2})^2$. The three two-prime products are one and $R_{2,3,11}=e^{\pi i/3}$.
\item $D=-163$, $M=41\cdot43\cdot47$, $\beta=21239$. Writing $w=\sqrt{-163}$, the three singleton values are
\[
 R_{41}=\bigl((1+w)/2\bigr)^2,\qquad
 R_{43}=\bigl((-3+w)/2\bigr)^2,\qquad
 R_{47}=\bigl((-5+w)/2\bigr)^2.
\]
Here $R_{41,43}=R_{43,47}=R_{41,43,47}=1$ and $R_{41,47}=e^{-2\pi i/3}$. The primes are $n^2+n+41$ for $n=0,1,2$.
\end{enumerate}
\end{example}
\begin{proof}
Write $\zeta=e^{2\pi i/24}$ and let $\pi_p$ be the displayed singleton base (before squaring in (iv), (v)). With principal square roots, the exponents $b_d$ below satisfy
\[
 \eta(d\tau)=\eta(\omega)\,\zeta^{b_d}\!\prod_{p\mid M/d}\sqrt{\pi_p},
 \qquad \omega=(\sqrt D-(D\bmod2))/2 .
\]
Both lists in each row are in increasing divisor order. The last row also gives the data for $D=-8$, $M=6$, $\beta=4$, with the generators for $2,3$ from (iv).
\[
\begin{array}{r|l|l}
 M & m_d\pmod{24} & b_d\pmod{24}\\\hline
 10&6,0,3,21&21,21,21,21\\
 22&8,8,2,20&20,20,20,20\\
 91&13,23,5,15&15,3,3,15\\
 66&11,10,10,23,4,23,23,16&5,4,22,23,22,23,17,16\\
 82861&19,15,19,15,15,19,9,13&19,21,7,3,9,13,3,13\\
 6&5,1,2,22&23,1,20,22
\end{array}
\]
To verify a row, start at $z=d\tau$ and translate to $z-n$ with $n=\lfloor\Re z+1/2\rfloor$, recording $\zeta^n$. If $|z|<1$, or $|z|=1$ and $\Re z>0$, record $\zeta^3\sqrt{1/z}$ at the translated $z$, invert to $-1/z$ and repeat. Otherwise stop. These are the two transformation laws of $\eta$. The final argument is $\omega$ in every entry, and the recorded exponents sum to $m_d$. Append $\sqrt{1/\pi_p}$ for $p\mid M/d$ to the recorded radicals. To multiply them as $s\sqrt H$, start with $s=H=1$. Set $H'=Hh$ for the next radicand $h$. Reverse $s$ if both principal arguments of $H,h$ are positive and $\arg H'\le0$, or if both are negative and $\arg H'>0$, then set $H=H'$. Thus only signs of rational coordinates are needed. At the end $H=\zeta^{2j}$, $-6<j\le6$, and $b_d\equiv m_d+j+12\,\mathbf1_{s=-1}\pmod{24}$.

For instance, at $M=10,d=1$, inversion gives $3+i$, and translation gives $i$. The factors yield $m_1=6$ and $H=(-3-i)/((1+i)(1+2i))=i$. The sign is negative, since the arguments of $-3-i$ and $1/(1+i)$ already sum to less than $-\pi$. Thus $b_1=6+3+12=21$. For each $p\mid M$, the sum $\sum_{d\mid M,\,p\nmid d}\chi_T(d)$ is $2^{k-1}$ for $T=\{p\}$ and zero otherwise. The remaining exponent $\sum_{d\mid M}\chi_T(d)b_d$ is zero modulo 24 in (i)--(v) except for $4$ at $(66,\{2,3,11\})$ and $16$ at $(82861,\{41,47\})$. This gives every stated value.
\end{proof}
In (iv), \(e^{\pi i/3}\notin\Q(\sqrt{-2})\). Example (v) also gives a value of \(R_T\) outside \(\Q(\sqrt{-163})\). Theorem~\ref{thm:hhclass1} places their minimal modular powers in the respective fields.

\begin{example}[Discriminant $-23$]\label{ex:hhcubic}
Take $M=6$, $\beta=1$, and $\tau=(-1+\sqrt{-23})/12$. The reduced forms are $[1,1,6]$ and $[2,\pm1,3]$. Thus $c=[2,1,3]=P_2$ generates $\Cl(-23)$ of order three, $P_3=[3,1,2]=c^{-1}$ and $P_6=1$. If $\theta$ is the positive real root of $x^3-x-1$, then
\[
 R_{2,3}(\tau)=\frac1\theta,\qquad e^{\lambda(c)-\lambda(1)}=\theta,\qquad
 |R_{2,3}(\tau_X)|=\sqrt{\theta}\quad(X=c,c^{-1}).
\]
\end{example}
\begin{proof}
Conjugation of forms gives $\lambda(c)=\lambda(c^{-1})$, and \eqref{eq:hhpoint} gives $\log|R_{2,3}|=\lambda(1)-\lambda(c)$ on the principal class and $-\tfrac12(\lambda(1)-\lambda(c))$ on the other two classes. Write $R=R_{2,3}$. The identity $\eta(-1/z)=\sqrt{-iz}\,\eta(z)$ gives $R(-1/(6z))=R(z)$ throughout $\Hh$. Also $R(-\bar z)=\overline{R(z)}$ by the defining product. On the connected semicircle $|z|=1/\sqrt6$ in $\Hh$, $R(z)$ is therefore real, nowhere zero, and positive at $i/\sqrt6$. In particular $R(\tau)>0$.

By Proposition~\ref{prop:hhalg}, $u=R(\tau)^{12}$ is a unit of the Hilbert class field $L$ of $\Q(\sqrt{-23})$. To identify its real units, note that $x^3-x-1$ is irreducible and has discriminant $-23$. Its splitting field has group $S_3$ and quadratic subfield $\Q(\sqrt{-23})$. Only $23$ divides the polynomial discriminant. Its inertia group is tame since $23\nmid6$, hence cyclic. Its nontrivial action on the ramified quadratic subfield forces it to be generated by a transposition. The cubic extension of the quadratic subfield is unramified. Its degree is the class number, proving that this splitting field is $L$. Its elements fixed by complex conjugation form $\Q(\theta)$.

The least real unit of $\Q(\theta)$ exceeding one is $\theta$. Indeed, if $1<v<\theta$ were a unit, its nonreal conjugates would be $z,\bar z$ with $|z|=v^{-1/2}$. Its integral trace $t$ lies in $\{0,1,2,3\}$ since $v-2/\sqrt v>-1$ and $v+2/\sqrt v<4$. The second coefficient of its monic cubic is $s=tv-v^2+1/v$. For each such $t$ the right-hand side is strictly decreasing for $v>1$, since $2v+1/v^2>3$. At $v=1$ it is $t$ and at $v=\theta$ it is $t\theta-1$, using $\theta^{-1}=\theta^2-1$. The open interval $(t\theta-1,t)$ contains no integer. This contradicts $s\in\Z$. Since $\theta$ is a unit, multiplication by its integral powers now shows that every positive unit is a power of $\theta$.

To determine the integer power, write $q=e^{2\pi i\tau}$. Cancellation gives
\[
 R(\tau)=e^{2\pi i\tau/12}\prod_{n\ge1}\frac{1+q^{3n}}{1+q^n}.
\]
Here $q=r e^{-\pi i/6}$ with $r=e^{-\pi\sqrt{23}/6}$. If $V$ is the modulus after retaining $n\le3$, then $|1+q^j|^2=1+2r^j\cos(j\pi/6)+r^{2j}$ gives
\[
 V^{12}=r\left(\frac{(1-r^6)^2(1+r^{18})}
 {(1+\sqrt3\,r+r^2)(1+r^2+r^4)}\right)^6 .
\]
Five alternating terms and the next-term error in $\pi=16\arctan(1/5)-4\arctan(1/239)$ give $333/106<\pi<355/113$. With $4.7958<\sqrt{23}<4.7959$, this gives $0.08117<r<0.08119$ by bounding $e^t$ for $t=\pi\sqrt{23}/6$ through degree $14$, with tail at most $t^{15}/(15!(1-t/16))$. Bounding each positive factor in the expression for $V^{12}$ with these endpoints and $1.732<\sqrt3<1.733$ gives $0.7548^{12}<V^{12}<0.7549^{12}$. The difference $|\log R(\tau)-\log V|$ is at most
\[
 \frac{r^4}{(1-r)(1-r^4)}
 +\frac{r^{12}}{(1-r^3)(1-r^{12})}<0.00005 ,
\]
using $|\log|1+w||\le |w|/(1-|w|)$. Thus $0.754<R(\tau)<0.756$. Evaluating $x^3-x-1$ at $1.324,1.325$ brackets $\theta$; raising the bounds to integer powers gives $\theta^{-13}<0.754^{12}<0.756^{12}<\theta^{-11}$. Since $R(\tau)^{12}$ is an integral power of $\theta$, the only possible exponent is $-12$. The remaining assertions follow from the class logarithms already computed.
\end{proof}

\subsection{Values of general modular units}
Let $r=(r_d)\in K_M$ give the modular unit $f_r$ normalized to leading coefficient one. Put $x_p=\sum_{p\mid d}r_d$ and $t_T=2^{-k}\sum_{d\mid M}\chi_T(d)r_d$. Fourier inversion gives $r_d=\sum_{T\ne\varnothing}t_T\chi_T(d)$, $|f_r|=\prod_T|R_T|^{t_T}$, and $t_{\{p\}}=-2^{1-k}x_p$.
\begin{corollary}\label{cor:hharbitrary}
Under the discriminant hypotheses of Theorem~\ref{thm:hhclosed}, suppose $\Cl(D)$ has exponent at most two. Then
\[
 |f_r(\tau_X)|
 =\prod_{p\mid M}p^{-x_p/4}
  \prod_{\substack{\psi\ne1\\T(\psi)\ne\varnothing}}
     U(\psi)^{-\psi(X)2^{k-1}t_{T(\psi)}/h(D)} .
\]
For every admissible order of class number one, with arbitrary conductor, $|f_r(\tau_X)|=\prod_{p\mid M}p^{-x_p/4}$.
\end{corollary}
\begin{proof}
Insert \eqref{eq:hhclosed} in the product for $|f_r|$. The singleton constants contribute $e^{\delta_{\{p\}}t_{\{p\}}}=p^{-x_p/4}$. For any admissible class-number-one order the same conclusion follows directly from Theorem~\ref{thm:hhclass1}.
\end{proof}
For example, $f_r=\eta(\tau)^5\eta(3\tau)/(\eta(2\tau)\eta(6\tau)^5)=R_2^3R_3^2$ has $x_2=-6$, $x_3=-4$ and $t_{\{2,3\}}=0$. The cusp formula gives a single simple pole at infinity, so $f_r$ is a Hauptmodul on $X_0(6)$. Its modulus is $2^{3/2}\cdot3=6\sqrt2$ at every oriented Heegner point of discriminant $-8,-15,-20$ or $-24$. For the three discriminants of class number two, $T(\psi)=\{2,3\}$. In particular $f_r(i/\sqrt6)=6\sqrt2$ by positivity. At $\tau=(-2+w)/6$, $w=\sqrt{-2}$, the oriented generators are $w,1-w$, since $2+w=w(1-w)$. Theorem~\ref{thm:hhclass1} gives $f_r(\tau)^{24}=(w^3(1-w)^2)^{24}$. The $M=6$ row in Example~\ref{ex:hhphases} gives phase exponent $5\cdot23-1+20-5\cdot22=24$, hence $f_r(\tau)=w^3(1-w)^2=-8+2w$.

\section{Values at two Fricke fixed points}\label{sec:two-fricke}

\subsection{The corresponding quadratic forms}
The point $i/\sqrt M$ corresponds to the principal form $[M,0,1]$ of
discriminant $-4M$ with $\beta=0$. Here $P_d$ is the class of
$[d,0,M/d]$. For odd $M$ the point $(1+i/\sqrt M)/2$ has the same
image on $X_0(M)$ as $(-1+i/\sqrt M)/2$. The latter has the following
Heegner descriptions.
\begin{itemize}
\item If $M\equiv3\pmod4$, the form is $[M,M,(M+1)/4]$, of
discriminant $-M$ with $\beta=M$. Its class is principal, and $P_d$
is the class of $[d,d,(d+M/d)/4]$.
\item If $M\equiv1\pmod4$, the form is $[2M,2M,(M+1)/2]$, of
discriminant $-4M$ with $\beta=0$. Its class is
$X'=[2,2,(M+1)/2]$.
\end{itemize}
The usual Fricke matrix fixes $i/\sqrt M$. For odd $M$, the matrix
$\left(\begin{smallmatrix}-M&-(M+1)/2\\2M&M\end{smallmatrix}\right)$ represents $W_M$ and fixes $(-1+i/\sqrt M)/2$.
Ramification gives $P_d^2=1$, and $P_1=P_M=1$. For $1<d<M$ neither
displayed form represents $1$. In the nondiagonal case write it as
$d(x+y/2)^2+(M/d)y^2/4$ and distinguish $y=0$, odd $y$, and nonzero even $y$.
Thus the kernel of $d\mapsto P_d$ consists of $1,M$, and $H=\{P_d:d\mid M\}$ has order $2^{k-1}$.
When $M\equiv1\pmod4$, none of $[d,0,M/d]$ represents $2$, whereas $X'$ does, so $X'\notin H$.

Write
\[
 \varpi_T(M)=R_T(i/\sqrt M)>0,\qquad
 \varpi'_T(M)=\left|R_T\!\left(\frac{1+i/\sqrt M}{2}\right)\right|
 \quad(M\text{ odd}).
\]
For odd $d$ and $y>0$, $\eta((d+iy)/2)$ is $e^{\pi i d/24}$ times
a positive number. It follows that the phase of the second value is
\[
 R_T\!\left(\frac{1+i/\sqrt M}{2}\right)
   =\exp\!\left(\frac{\pi i(-1)^{|T|}\Lambda_T}{24}\right)\varpi'_T(M).
\]
In Ramanujan's notation $f(-q)=\prod_{n\ge1}(1-q^n)$, the products read
\begin{align*}
 \varpi_T(M)&=q^{(-1)^{|T|}\Lambda_T/24}
      \prod_{d\mid M}f(-q^d)^{\chi_T(d)}, & q&=e^{-2\pi/\sqrt M},\\
 \varpi'_T(M)&=q^{(-1)^{|T|}\Lambda_T/24}
      \prod_{d\mid M}f(q^d)^{\chi_T(d)}, & q&=e^{-\pi/\sqrt M}.
\end{align*}

\begin{theorem}\label{thm:two-fricke}
Let $k\ge2$. Every statement involving $\varpi'_T$ assumes $M$ odd.
\begin{enumerate}
\item[(a)] For odd $|T|$ the first and, when defined, second Fricke
values equal $p^{2^{k-3}}$ if $T=\{p\}$ and equal $1$ if $|T|\ge3$.
\item[(b)] Let $|T|$ be even. If \eqref{eq:hhregular} holds for $D=-4M$, then
\[
 \varpi_T(M)=\prod_{T(\psi)=T}U(\psi)^{-2^{k-1}/h(-4M)} .
\]
Under the same hypothesis at $D=-4M$, for $M\equiv1\pmod4$ one also has
\[
 \varpi'_T(M)=\prod_{T(\psi)=T}U(\psi)^{-\psi(X')2^{k-1}/h(-4M)}.
\]
If $M\equiv3\pmod4$ and \eqref{eq:hhregular} holds for $D=-M$, then
$\varpi'_T(M)=\prod_{T(\psi)=T}U(\psi)^{-2^{k-1}/h(-M)}$.
\item[(c)] Fix a nonempty even $T$. Suppose $\Cl(-4M)$ has exponent at most two. If
$M\not\equiv1\pmod4$, there is exactly one genus character $\psi=\psi_T$ with $T(\psi)=T$, and
\[
 \varpi_T(M)=U(\psi_T)^{-1}
     =\eps_{D_1}^{-2c_\psi h(D_1)h(D_2)/w(D_2)}
\]
for the factorization $D_1D_2$ attached to $\psi$. If $M\equiv1\pmod4$,
there are exactly two genus characters $\psi_T^\pm$ with $T(\psi_T^\pm)=T$,
chosen so that $\psi_T^\pm(X')=\pm1$. Then
\[
 \varpi_T(M)\varpi'_T(M)=U(\psi_T^+)^{-1},\qquad
 \varpi_T(M)/\varpi'_T(M)=U(\psi_T^-)^{-1}.
\]
For $M\equiv3\pmod4$ the corresponding assertion at discriminant $-M$
requires only that $\Cl(-M)$ have exponent at most two. Each nonempty
even $T$ then has exactly one genus character $\psi'_T$ with $T(\psi'_T)=T$,
and $\varpi'_T(M)=U(\psi'_T)^{-1}$.
\end{enumerate}
\end{theorem}

\begin{proof}
Part (a) follows from Corollary~\ref{cor:hhpair}(d), and (b) from Theorem~\ref{thm:hhclosed}.
The genus number for $-4M$ is $2^{k-1}$ for $M\not\equiv1\pmod4$
and $2^k$ for $M\equiv1\pmod4$. For $-M$ with $M\equiv3\pmod4$
it is $2^{k-1}$ \cite[\S\S3, 6]{cox}. With one class per genus this
is the class number, and $[\Cl(D):H]$ is respectively $1$ or $2$.
The character $d\mapsto\chi_T(d)$ descends to $H$ exactly when
$\chi_T(M)=1$, equivalently $|T|$ is even. Restriction of characters
to $H$ is surjective and each fibre has size $[\Cl(D):H]$.
At index two the two extensions differ by the unique nontrivial
character of $\Cl(D)/H$, which takes value $-1$ at $X'$.
The formulas now follow from (b).
\end{proof}

\begin{remark}[M\"obius products]\label{rem:mob-fricke}
For $T=P$, $\chi_P=\mu$ and $\Lambda_P=\varphi(M)$.
When $k\ge3$ is odd, Theorem~\ref{thm:two-fricke}(a) gives
\begin{align*}
 \prod_{d\mid M} f(-q^d)^{\mu(d)}&=q^{\varphi(M)/24},
       & q&=e^{-2\pi/\sqrt M},\\
 \prod_{d\mid M} f(q^d)^{\mu(d)}&=q^{\varphi(M)/24},
       & q&=e^{-\pi/\sqrt M}\quad(M\text{ odd}).
\end{align*}
When $k$ is even the products are $q^{-\varphi(M)/24}\varpi_P(M)$
and $q^{-\varphi(M)/24}\varpi'_P(M)$, respectively, as in (1.6).
\end{remark}

For the first Fricke point, part (c) specializes the Kronecker--Pell
method for Ramanujan's class invariants \cite{bcz} to the thirty
known squarefree idoneal levels with $k\ge2$,
\[
\begin{gathered}
6,\ 10,\ 15,\ 21,\ 22,\ 30,\ 33,\ 42,\ 57,\ 58,\ 70,\ 78,\ 85,\ 93,\
102,\ 105,\\
130,\ 133,\ 165,\ 177,\ 190,\ 210,\ 253,\ 273,\ 330,\ 345,\ 357,\ 385,\
462,\ 1365 .
\end{gathered}
\]
At the second point with $M\equiv3\pmod4$, the list contains twenty-four levels, from $15$ to $3315$.
Here is the finite count. Run through squarefree $2\le M\le5460$ with $k\ge2$ and $|D_K|\le5460$, where $D_K=-M$ for $M\equiv3\bmod4$ and $D_K=-4M$ otherwise. For $D<0$, count one for each pair
\[
 1\le a\le\lfloor\sqrt{|D|/3}\rfloor,\quad -a<b\le a,\quad
 c=(b^2-D)/(4a)\in\Z,\quad c\ge a,\quad \gcd(a,b,c)=1,
\]
omitting $b<0$ when $a=c$. Reduction gives $h(D)$ since these are precisely the primitive reduced forms, counted once per proper class. Compare $h(-4M)$, and also $h(-M)$ when $M\equiv3\bmod4$, with the genus number ($2^k$ if $M\equiv1\bmod4$, $2^{k-1}$ otherwise). For $k=2,3,4$ the numbers tested are $789,362,40$. Equality gives $13,13,4$ first-point levels and $11,9,4$ second-point levels of residue $3\bmod4$. These are the levels in the appendix. For conductor two, $\Cl(-4M)$ surjects onto $\Cl(-M)$, and the field determines $M$. Weinberger's theorem therefore allows at most one further level in either list and gives completeness under the generalized Riemann hypothesis \cite{weinberger}. The listed evaluations are unconditional.

\begin{example}[Level $66$]\label{ex:fricke66}
At $M=66$ the reduced forms of discriminant $-264$ have leading coefficients
$1,2,3,5,5,6,7,7$. Hence $h(-264)=8$. The genus number is $4$, so
$\Cl(-264)\cong\Z/4\Z\times\Z/2\Z$ and $H=\Cl(-264)[2]$.
For the three squares below, lift the middle coefficient $b$ to $B\equiv b\bmod 10$ with $B^2\equiv D\bmod100$. Since $5\nmid b$, Gauss composition gives $[25,B,(B^2-D)/100]$. Thus $[5,4,14]^2=[25,44,22]\sim[3,0,22]=P_{22}$,
so $\Cl(-264)^2=\{1,P_{22}\}$. The only
nonempty even $T$ with $\chi_T(22)=1$ is $\{2,11\}$. The two
genus characters with $T(\psi)=\{2,11\}$ correspond to
$-264=24\cdot(-11)=88\cdot(-3)$. Thus
\[
 \varpi_{2,11}(66)=\eps_{24}^{-1/2}\eps_{88}^{-1/6}
   =(5+2\sqrt6)^{-1/2}(197+42\sqrt{22})^{-1/6},
\]
which is (1.7). At the first point take $M=114,138,154,238,258,282,310,322$ or
$390,510,570,690$, and at the second point take $M=651,915,987$.
The same reduced-form count gives $h(D)=8$ for the three-prime levels and $16$ for the four-prime levels.
The genus number is $|H|=2^{k-1}$, so $H=\Cl(D)[2]$ and
$|\Cl(D)^2|=2$. Exactly $2^{k-2}-1$ nontrivial characters of $H$
are trivial on this subgroup. Thus one nonempty even $T$ satisfies \eqref{eq:hhregular}
at each three-prime level and three at each four-prime level. At $114$ and $154$ the respective nonidentity squares are
$[5,-2,23]^2=[25,38,19]\sim[6,0,19]=P_6$ and
$[5,-2,31]^2=[25,28,14]\sim[11,0,14]=P_{11}$. Thus $T=\{2,3\}$ and $\{2,7\}$, respectively. Formula~\eqref{eq:hhgenusprime} selects $24(-19),152(-3)$ and $56(-11),77(-8)$, giving
\[
 \varpi_{2,3}(114)=\eps_{24}^{-1/2}\eps_{152}^{-1/6},\qquad
 \varpi_{2,7}(154)=\eps_{56}^{-1/2}\eps_{77}^{-1/2}.
\]
At level $66$, $\chi_{\{2,3\}}(22)=-1$, and $\varpi_{2,3}(66)$ is a
degree-four unit with no positive power in an abelian number field,
as shown after Corollary~\ref{cor:fricke-traces}.
\end{example}

\begin{example}\label{ex:fricke-values}
With $\phi=\eps_5$, the appendix gives
\begin{align*}
 \varpi_{2,3}(30)&=\phi^{-2},&
 \varpi_{2,5}(30)&=(3+\sqrt{10})^{-2/3},&
 \varpi_{3,5}(30)&=(1+\sqrt2)^{-2},\\
 \varpi_{2,7}(42)&=\eps_{21}^{-1},&
 \varpi_{3,7}(42)&=(5+2\sqrt6)^{-1},&
 \varpi_{2,3}(42)&=(15+4\sqrt{14})^{-1/3},\\
 \varpi_{2,3,5,7}(210)&=\phi^{-4},&
 \varpi_{2,3,7,11}(462)&=\eps_{21}^{-2},&
 \varpi_{3,11}(462)&=(1+\sqrt2)^{-12}.
\end{align*}
Here $\eps_{21}=(5+\sqrt{21})/2$. At the second point,
\begin{align*}
 \varpi'_{5,7}(35)&=\phi^{-1},&
 \varpi'_{7,13}(91)&=((3+\sqrt{13})/2)^{-1},\\
 \varpi'_{5,23}(115)&=\phi^{-3},&
 \varpi'_{3,7}(1155)&=\phi^{-12}.
\end{align*}
The level $210$ value gives (1.6). In Ramanujan's notation the
identities for levels $6,10,35$ read
\begin{align*}
 q^{1/12}\frac{f(-q)f(-q^6)}{f(-q^2)f(-q^3)}
    &=(1+\sqrt2)^{-1/3}, &q&=e^{-2\pi/\sqrt6},\\
 q^{1/6}\frac{f(-q)f(-q^{10})}{f(-q^2)f(-q^5)}
    &=(\sqrt5-1)/2, &q&=e^{-2\pi/\sqrt{10}},\\
 q\frac{f(q)f(q^{35})}{f(q^5)f(q^7)}
    &=(\sqrt5-1)/2, &q&=e^{-\pi/\sqrt{35}}.
\end{align*}
For $M\equiv1\pmod4$ part (c) gives (1.8) and, for example,
\begin{align*}
 \varpi_{3,11}(33)\varpi'_{3,11}(33)&=(23+4\sqrt{33})^{-1/2},\\
 \varpi_{5,17}(85)\varpi'_{5,17}(85)&=(4+\sqrt{17})^{-2},&
 \varpi_{5,17}(85)/\varpi'_{5,17}(85)&=\phi^{-4}.
\end{align*}
\end{example}

\subsection{Class invariants}
Let $\ell\mid M$ be prime and put $n=M/\ell$. Define
\[
 \rho_\ell(M)=\ell^{-1/4}\frac{\eta(i\sqrt{n/\ell})}{\eta(i\sqrt{n\ell})},
 \qquad
 \rho'_\ell(M)=\ell^{-1/4}
   \frac{\eta((1+i\sqrt{n/\ell})/2)}{\eta((1+i\sqrt{n\ell})/2)}
 \quad (M\text{ odd}).
\]
Both eta factors in the primed quotient have phase $e^{\pi i/24}$.

\begin{corollary}\label{cor:class-invariants}
For $k\ge2$,
\[
 \rho_\ell(M)^{2^{k-1}}=\prod_{\substack{\ell\in T\\ |T|\ {\rm even}}}
       \varpi_T(M)^{-1},\qquad
 \rho'_\ell(M)^{2^{k-1}}=\prod_{\substack{\ell\in T\\ |T|\ {\rm even}}}
       \varpi'_T(M)^{-1}\quad(M\text{ odd}).
\]
In particular, $g_M=\rho_2(M)$ for even $M$, while
$\lambda_n=\rho'_3(3n)^6$ for squarefree $n$ prime to $6$. Hence
\[
 g_M^{2^{k-1}}=\prod_{\substack{2\in T\\ |T|\ {\rm even}}}\varpi_T(M)^{-1},
 \qquad
 \lambda_n^{2^{k-2}}=\prod_{\substack{3\in T\\ |T|\ {\rm even}}}
       \varpi'_T(3n)^{-3},
\]
with $M=3n$ in the second formula. The quantities $\rho_\ell(\ell n)$ and
$\rho'_\ell(\ell n)$ are Yi's $r_{\ell,n}$ and $r'_{\ell,n}$ \cite[(1.11)--(1.12)]{yi}.
In particular $\varpi_{p,q}(pq)=r_{p,q}^{-2}$.
\end{corollary}

\begin{proof}
Fourier inversion gives
$\log\eta(id/\sqrt M)=2^{-k}\sum_{T\subseteq P}
 \chi_T(d)\log\varpi_T(M)$, including the product for $T=\varnothing$.
Subtract the formula for $d=M$ from that for $d=n$. The coefficient
$\chi_T(n)-\chi_T(M)$ equals $2(-1)^{|T|-1}$ for $\ell\in T$ and
vanishes otherwise. By Theorem~\ref{thm:two-fricke}(a), the odd sets
contribute $\tfrac14\log\ell$, and the remaining contribution is
$-2^{1-k}\sum_{\ell\in T,\ |T|\ {\rm even}}\log\varpi_T(M)$.
For odd $M$ the same argument applies to absolute values at the
second point, since $d$ is odd and
$|\eta(d(1+i/\sqrt M)/2)|=|\eta((1+id/\sqrt M)/2)|$.
In the convention of \cite{bcz}, $g_M=2^{-1/4}q^{-1/24}\prod_{j\ge1}(1-q^{2j-1})$ for $q=e^{-\pi\sqrt M}$. Cancelling the even factors in $\rho_2(M)$ gives this product. With $q=e^{-\pi\sqrt{n/3}}$, Ramanujan's $\lambda_n=3^{-3/2}q^{-1/2}f(q)^6/f(q^3)^6$ \cite[(1.4)]{bckz}. The eta products in $\rho'_3(3n)^6$ give precisely these factors. Yi's identifications are the definitions cited above.
\end{proof}

Together with the appendix this yields, for example,
\begin{align*}
 g_{70}&=\phi(1+\sqrt2)^{1/2},&
 g_{130}&=\phi^{3/2}((3+\sqrt{13})/2)^{1/2},&
 g_{190}&=\phi^{3/2}(3+\sqrt{10})^{1/2},\\
 \lambda_{17}&=4+\sqrt{17},&
 \lambda_{41}&=32+5\sqrt{41},&
 \lambda_{89}&=500+53\sqrt{89}.
\end{align*}
These values are recorded in \cite{berndt5} and \cite[pp.~267--268]{bckz}.
The same formulas give $\lambda_5=\phi$ and
$\lambda_{65}=\phi^6(8+\sqrt{65})$. Indeed, the two entries containing $3$ at $M=195$ give $\lambda_{65}^2=(\eps_{65}^{-2/3}\phi^{-4})^{-3}$, and all factors are positive.

\section{Genus-zero levels and trace identities}\label{sec:signed-traces}

\subsection{Traces}
Let $\Gamma_0(M)^+$ be generated by $\Gamma_0(M)$ and all $W_Q$, and
let $X_0(M)^+$ be the quotient curve. The Atkin--Lehner group acts
simply transitively on the cusps of $X_0(M)$. Thus $X_0(M)^+$ has
one cusp. Its stabilizer has the same width as at $\infty$ on $X_0(M)$, so $q$ is a local parameter.
Recall that $c_T=p^{e_T2^{k-2}}$ for $T=\{p\}$ and $c_T=1$ otherwise.
By Proposition~\ref{prop:al-law}, $s_T$ is trivial whenever $M$ is even or $k\ge3$.
Write $A_T$ for the invariant in
part (a) or (c) of the following theorem.

\begin{theorem}\label{thm:signed-traces}
Let $k\ge2$.
\begin{enumerate}
\item[(a)] If $s_T=1$, $A_T=u_T+c_Tu_T^{-1}$ is invariant
under $\Gamma_0(M)^+$, holomorphic on $\Hh$, and has an integral
Laurent expansion beginning with $q^{-m_T}$.
\item[(b)] Suppose that $m_{T_0}=1$ for a set $T_0$ covered by (a)
or (c). Then $X_0(M)^+$ has genus zero, and its normalized Hauptmodul is
\[
 j_M=A_{T_0}-(-1)^{|T_0|}e_{T_0}=q^{-1}+O(q).
\]
For every $T$ covered by (a) or (c) there is a monic polynomial
$F_T\in\Z[x]$ of degree $m_T$ such that $A_T=F_T(j_M)$.
If $m_T=1$, then $A_T=j_M+(-1)^{|T|}e_T$.
\item[(c)] If $M=pq$ is odd, $T=\{p,q\}$ and $s_T(p)=-1$, the
invariant is $A_T=u_T^{-1}-u_T$. It is holomorphic on $\Hh$, and its
integral Laurent expansion also begins with $q^{-m_T}$.
\end{enumerate}
There is some $T$ with $m_T=1$, equivalently $\Lambda_T\mid24$,
precisely at the levels
\[
 6,\ 10,\ 14,\ 15,\ 21,\ 22,\ 26,\ 30,\ 33,\ 35,\ 39,\ 42,\
 46,\ 66,\ 70,\ 78 .
\]
Every level on the list has a set $T_0$ as in (b).
\end{theorem}

\begin{proof}
When $s_T=1$, Proposition~\ref{prop:al-law} shows that $W_Q$ either fixes $u_T$
or sends it to $c_Tu_T^{-1}$. Both actions fix $A_T$. In case (c),
$W_p$ and $W_q$ send $u_T$ to $-u_T^{-1}$, and $W_M$ fixes $u_T$.
It follows that $u_T^{-1}-u_T$ is invariant. In both cases holomorphy
on $\Hh$ follows from the unit property. By Proposition~\ref{prop:minimal}, the pole
comes from $u_T$ for odd $|T|$ and from $u_T^{-1}$ for even $|T|$.
The Laurent expansion begins with coefficient $1$ and is integral.

When $m_{T_0}=1$, the function $A_{T_0}$ has just one simple pole on
the compact quotient. Its map to $\mathbb P^1$ has degree one. Its
constant coefficient is $(-1)^{|T_0|}e_{T_0}$ by Proposition~\ref{prop:minimal},
giving the stated normalization of $j_M$. Any $A_T$ has only a pole
at that cusp, hence is a polynomial in $j_M$ of degree $m_T$.
Successively canceling leading Laurent coefficients with powers of
$j_M$ gives integral coefficients and leading coefficient $1$.
The constant term settles the case $m_T=1$.

For the finite enumeration, consider the primes $2,3,5,7,11,13,23$ in turn.
For each prime $p$, omit it, include it in $T$ with factor $p-1$, or include it in $P\setminus T$ with factor $p+1$. Discard a product as soon as it fails to divide $24$. These choices are exhaustive since $p-1\mid24$ or $p+1\mid24$ for every prime used. Keep $T\ne\varnothing$ and $|P|\ge2$. In the displayed order of the levels the numbers of surviving sets $T$ are
\[
 3,\ 3,\ 2,\ 2,\ 2,\ 1,\ 1,\ 4,\ 1,\ 1,\ 1,\ 2,\ 1,\ 1,\ 1,\ 1 .
\]
A choice of $T_0$ is $\{2\}$ at $6,10,14,22,30,46$, $\{3\}$ at $15,33$, $\{7\}$ at $21$, and $\{2,3\}$ at $66$. Use the full prime set at the other six levels. Proposition~\ref{prop:al-law} gives trivial sign for these choices except at $35$, where (c) applies.
\end{proof}

Table~\ref{tab:fricke-traces} lists the polynomials at the seven levels
with several simple-pole sets. For each row put $m=m_T$. Collecting the factors of the Euler product with the same power of $q$ gives
\[
 q^m A_T\equiv\prod_{n=1}^{m}(1-q^n)^{
 (-1)^{|T|+1}e_T\sum_{d\mid\gcd(n,M)}\chi_T(d)}
 \pmod {q^{m+1}} .
\]
The other term of the trace starts at $q^m$ before multiplication by $q^m$. With $j=j_M$ from the chosen $T_0$, subtract multiples of $j^m,j^{m-1},\ldots,1$ in order to cancel the coefficients of $q^{-m},\ldots,q^0$. Theorem~\ref{thm:signed-traces}(b) proves the resulting identity. The coefficients of $j$ needed for this subtraction follow by retaining $dn\le m$ in the eta products of the simple-pole trace, including its reciprocal term. For $j=q^{-1}+a_1q+a_2q^2+\cdots$, retaining $dn\le3$ gives
\[
 \begin{array}{c|rrrrrrr}
 M &6&10&14&15&21&30&42\\ \hline
 a_1&79&22&11&8&6&4&2\\
 a_2&352&56&20&22&6&2&2
 \end{array}
 \,.
\]
Only the quartic row needs $[q^3]j$. At $M=42$ the full prime set gives $u_P=q^{-1}(1-q)^2+O(q^4)$ and $j=u_P+u_P^{-1}+2=q^{-1}+2q+2q^2+3q^3+O(q^4)$. For $T=\{2\}$ the product above reduces modulo $q^5$ to
\[
 (1-q)^3(1-q^3)^6\equiv1-3q+3q^2-7q^3+18q^4 .
\]
Cancellation gives $A_2=j^4-3j^3-5j^2+3j+20$. The other rows use at most the first two coefficients in the array.
The groups $\Gamma_0(M)^+$ are those assigned to $6A,10A,14A,15A,21A,30B,42A$ in \cite[Table~2]{conway}. By the moonshine theorem \cite[Theorem~1.1]{borcherds} and uniqueness of the normalized Hauptmodul, these are their McKay--Thompson series. For $m_T=1$ several eta expressions for these series
are recorded in \cite{oeis}, including A007254, A058563 and A058613
at levels $6,21,30$. The theorem identifies the powers of \(R_T\)
in the expressions and determines their signs and constants.

\begin{table}[tbp]
\centering\small
\begin{tabular}{c l l l}
\hline
$M$&$T$&$A_T$&$F_T(j)$\\ \hline
6&$\{2\}$&$R_2^6+64R_2^{-6}$&$j-6$\\
 &$\{3\}$&$R_3^4+81R_3^{-4}$&$j-4$\\
 &$\{2,3\}$&$R_{2,3}^{12}+R_{2,3}^{-12}$&$j+12$\\
10&$\{2\}$&$R_2^4+16R_2^{-4}$&$j-4$\\
 &$\{5\}$&$R_5^2+25R_5^{-2}$&$j-2$\\
 &$\{2,5\}$&$R_{2,5}^6+R_{2,5}^{-6}$&$j+6$\\
14&$\{2\}$&$R_2^3+8R_2^{-3}$&$j-3$\\
 &$\{7\}$&$R_7^4+7^4R_7^{-4}$&$j^3-4j^2-35j+52$\\
 &$\{2,7\}$&$R_{2,7}^4+R_{2,7}^{-4}$&$j+4$\\
15&$\{3\}$&$R_3^2+9R_3^{-2}$&$j-2$\\
 &$\{3,5\}$&$R_{3,5}^{-3}-R_{3,5}^3$&$j+3$\\
21&$\{7\}$&$R_7+7R_7^{-1}$&$j-1$\\
 &$\{3,7\}$&$R_{3,7}^2+R_{3,7}^{-2}$&$j+2$\\
30&$\{2\}$&$R_2+4R_2^{-1}$&$j-1$\\
 &$\{3\}$&$R_3^2+81R_3^{-2}$&$j^3-2j^2-15j+18$\\
 &$\{5\}$&$R_5+25R_5^{-1}$&$j^2-j-10$\\
 &$\{2,3\}$&$R_{2,3}^2+R_{2,3}^{-2}$&$j+2$\\
 &$\{2,5\}$&$R_{2,5}^3+R_{2,5}^{-3}$&$j^2+3j-2$\\
 &$\{3,5\}$&$R_{3,5}+R_{3,5}^{-1}$&$j+1$\\
 &$\{2,3,5\}$&$R_{2,3,5}^3+R_{2,3,5}^{-3}$&$j-3$\\
42&$\{2\}$&$R_2^3+64R_2^{-3}$&$j^4-3j^3-5j^2+3j+20$\\
 &$\{3\}$&$R_3+9R_3^{-1}$&$j^2-j-6$\\
 &$\{7\}$&$R_7+49R_7^{-1}$&$j^3-j^2-8j-2$\\
 &$\{2,3\}$&$R_{2,3}^3+R_{2,3}^{-3}$&$j^2+3j+2$\\
 &$\{2,7\}$&$R_{2,7}+R_{2,7}^{-1}$&$j+1$\\
 &$\{3,7\}$&$R_{3,7}^2+R_{3,7}^{-2}$&$j^3+2j^2+j-2$\\
 &$\{2,3,7\}$&$R_{2,3,7}^2+R_{2,3,7}^{-2}$&$j-2$\\ \hline
\end{tabular}
\caption{Invariant traces in the normalized Hauptmodul $j=j_M$.
The row $M=15$, $T=\{3,5\}$ uses the signed difference in
Theorem~\ref{thm:signed-traces}(c). The singletons $\{5\}$ at level
$15$ and $\{3\}$ at level $21$ have nontrivial sign characters.}
\label{tab:fricke-traces}
\end{table}

Eliminating $j$ recovers identities of the type of Ramanujan's
$P$--$Q$ equations \cite[Ch.~25, Entries 51--72]{berndt4}. At levels
$15$ and $21$, write $\eta_d=\eta(d\tau)$. The two relations are
\begin{align}
 R_3^2+9R_3^{-2}+5&=R_{3,5}^{-3}-R_{3,5}^3,
 &R_3&=\frac{\eta_1\eta_5}{\eta_3\eta_{15}},&
 R_{3,5}&=\frac{\eta_1\eta_{15}}{\eta_3\eta_5},\tag{7.1}\label{eq:trace-15}\\
 R_7+7R_7^{-1}+3&=R_{3,7}^2+R_{3,7}^{-2},
 &R_7&=\frac{\eta_1\eta_3}{\eta_7\eta_{21}},&
 R_{3,7}&=\frac{\eta_1\eta_{21}}{\eta_3\eta_7}.\tag{7.2}\label{eq:trace-21}
\end{align}
At level $30$ the rows of Table~\ref{tab:fricke-traces} give (1.10), with
\begin{align*}
 R_2&=\frac{\eta_1\eta_3\eta_5\eta_{15}}{\eta_2\eta_6\eta_{10}\eta_{30}},&
 R_{3,5}&=\frac{\eta_1\eta_2\eta_{15}\eta_{30}}{\eta_3\eta_5\eta_6\eta_{10}},\\
 R_{2,3}&=\frac{\eta_1\eta_5\eta_6\eta_{30}}{\eta_2\eta_3\eta_{10}\eta_{15}},&
 R_{2,3,5}&=\frac{\eta_1\eta_6\eta_{10}\eta_{15}}{\eta_2\eta_3\eta_5\eta_{30}} .
\end{align*}
Further identities at $30$ and $42$, respectively, are
\begin{align*}
 R_5+25R_5^{-1}&=(R_2+4R_2^{-1})^2+(R_2+4R_2^{-1})-10,\\
 R_{2,7}+R_{2,7}^{-1}&=R_{2,3,7}^2+R_{2,3,7}^{-2}+3.
\end{align*}
For $P=\eta(3\tau)\eta(5\tau)/(\eta(\tau)\eta(15\tau))$ and
$Q(\tau)=P(2\tau)$ the relation between $R_{3,5}$ and $R_{2,3,5}$
in (1.10) takes the mixed form
\[
 PQ+(PQ)^{-1}=(Q/P)^3+(P/Q)^3+4 .
\]

\subsection{Hauptmodul values and Fricke values}
\begin{corollary}\label{cor:fricke-traces}
Write $J_M=j_M(i/\sqrt M)$. The rational and quadratic values of $J_M$
at the levels of Theorem~\ref{thm:signed-traces} are
\[
 \begin{array}{c|rrrrrrrr}
 M&6&10&15&22&30&42&70&78\\ \hline
 J_M&22&12&8&6&5&4&3&3
 \end{array}
\]
\[
 \begin{array}{c|cccccc}
 M&14&21&33&39&46&66\\ \hline
 J_M&3+4\sqrt2&1+2\sqrt7&1+2\sqrt3&
 (5+\sqrt{13})/2&1+2\sqrt2&(1+\sqrt{33})/2
 \end{array}
\]
The remaining $J_{26}$ and $J_{35}$ are the unique real roots of
$x^3-2x^2-15x-16$ and $x^3-2x^2-4x-20$, respectively. In particular,
\begin{align*}
 \varpi_{2,3}(6)^{12}+\varpi_{2,3}(6)^{-12}&=34,&
 \varpi_{2,5}(10)^6+\varpi_{2,5}(10)^{-6}&=18,\\
 \varpi_{3,5}(15)^{-3}-\varpi_{3,5}(15)^3&=11,&
 \varpi_{3,5}(30)+\varpi_{3,5}(30)^{-1}&=6,\\
 \varpi_{2,3}(30)^2+\varpi_{2,3}(30)^{-2}&=7,&
 \varpi_{2,7}(42)+\varpi_{2,7}(42)^{-1}&=5 .
\end{align*}
Five levels with class groups of exponent greater than two give
\begin{align*}
 \varpi_{2,7}(14)^4+\varpi_{2,7}(14)^{-4}&=7+4\sqrt2,\\
 \varpi_{3,13}(39)+\varpi_{3,13}(39)^{-1}&=(7+\sqrt{13})/2,\\
 \varpi_{2,3}(66)+\varpi_{2,3}(66)^{-1}&=(3+\sqrt{33})/2,\\
 \varpi_{2,13}(26)^2+\varpi_{2,13}(26)^{-2}&=y,&y^3-8y^2+5y-2&=0,\\
 \varpi_{5,7}(35)^{-1}-\varpi_{5,7}(35)&=z,&z^3-5z^2+3z-19&=0 .
\end{align*}
In particular $g_{14}^2+g_{14}^{-2}=1+\sqrt2$.
\end{corollary}

\begin{proof}
At $i/\sqrt M$, Theorem~\ref{thm:two-fricke}(a) gives
$A_{\{p\}}=2p^{e_{\{p\}}2^{k-3}}$ if $s_{\{p\}}=1$ and
$A_T=2$ for odd $|T|\ge3$. If such a $T$ has $m_T=1$, then
$J_M=A_T(i/\sqrt M)+e_T$. Substitution of the twelve odd-cardinality choices $T_0$ in the proof of Theorem~\ref{thm:signed-traces} gives the corresponding entries.

For the other four levels use the $T_0$ specified there. The finite products above give the following coefficients. The last row runs from the leading pole of $A_{\{p\}}$ through its constant term.
\[
\begin{array}{c|cccc}
 M&39&66&26&35\\
 p&13&2&13&7\\
 ([q]j_M,[q^2]j_M)&(3,1)&(2,0)&(4,4)&(1,4)\\
 A_{\{p\}}&(1,-1,-1)&(1,-1,0)&(1,-2,-3,6)&(1,-2,-1,2)
\end{array}
\]
Comparing coefficients at $39$ and $66$ gives $F_{13}(x)=x^2-x-7$
and $F_2(x)=x^2-x-4$, respectively, with
$A_{\{13\}}(i/\sqrt{39})=2\sqrt{13}$ and $A_{\{2\}}(i/\sqrt{66})=4$.
In each case $J_M+1=u_{T_0}+u_{T_0}^{-1}\ge2$ at the point, for
$T_0=\{3,13\}$ and $\{2,3\}$, respectively. This selects the positive
root. At $26$ and $35$, $F_{13}(x)=x^3-2x^2-15x+10$ and
$F_7(x)=x^3-2x^2-4x-6$, with respective Fricke traces $26$ and $14$.
The cubics $F_{13}(x)-26$ and $F_7(x)-14$ have discriminants $-1664$ and $-14000$, respectively, so each has a single real root.
Both cubics are irreducible modulo $3$, by testing $0,\pm1$.

The remaining statements follow from $A_T= j_M+e_T$ for the even
sets of pole order one. The differences at $15$ and $35$ are exactly
the signed traces of Theorem~\ref{thm:signed-traces}(c).
For $g_{14}$ use $g_{14}^2=\varpi_{2,7}(14)^{-1}$ from
Corollary~\ref{cor:class-invariants}. Taking positive square roots twice, using $(v+v^{-1})^2=v^2+v^{-2}+2$, gives the claim.
\end{proof}

For level $66$ set $v=\varpi_{2,3}(66)$ and
$a=v+v^{-1}=(3+\sqrt{33})/2$. Since $a^2-3a-6=0$, clearing
denominators gives
\[
 v^4-3v^3-4v^2-3v+1=0 .
\]
The other conjugate $a'=(3-\sqrt{33})/2$ belongs to $(-2,2)$.
At that real embedding the equation $t^2-a't+1=0$ has two nonreal
roots of absolute value one. Hence $t^2-at+1$ cannot split in the
totally real field $\Q(a)$, and $[\Q(v):\Q]=4$. If $v^n$ lay in an
abelian number field for some $n\ge1$, the real number $v^n$ would
have only real conjugates, since complex conjugation commutes with
every automorphism of that field. Its conjugate coming from a root
above $a'$ would then be real and of absolute value one, hence $\pm1$.
This would give $v^n=\pm1$, impossible since $v>0$ and $a>2$.

For discriminant $-4M$ each $p\mid M$ ramifies, so its orientation is unique.
Over $K=\Q(\sqrt{-M})$ the class group acts simply transitively on these Heegner
points, and $W_p$ translates by $P_p$ \cite[(4.2), (5.1)]{gross}.
The projected orbit has size $[\Cl(-4M):H]$. At the sixteen levels above,
the Hauptmodul is defined over $\Q$ by its rational expansion at the rational cusp $\infty$ on $X_0(M)$. It separates the quotient points, so the orbit size is
$[K(J_M):K]$. Since $J_M$ is real and $K$ is imaginary quadratic, it also
equals $[\Q(J_M):\Q]$. In particular $\Cl(-4M)=H$ at the eight rational
values in Corollary~\ref{cor:fricke-traces}. See also \cite{enge,chenyui}. Finally, the numbers $34,18,7,6,5$ are the sums
of the corresponding units and their reciprocals for $(1+\sqrt2)^4$,
$\phi^6$, $\phi^4$, $(1+\sqrt2)^2$, and $(5+\sqrt{21})/2$,
respectively. These give the same Fricke values by
Theorem~\ref{thm:two-fricke} and Corollary~\ref{cor:fricke-traces}.

\clearpage
\appendix
\section{Tables of Fricke values}\label{app:tables}\small
The subscript on \(\eps_d\) is the fundamental discriminant of the real quadratic field.
For example,
\[
\begin{aligned}
\eps_5&=(1+\sqrt5)/2, &\eps_8&=1+\sqrt2, &\eps_{12}&=2+\sqrt3, &\eps_{13}&=(3+\sqrt{13})/2,\\
\eps_{17}&=4+\sqrt{17}, &\eps_{21}&=(5+\sqrt{21})/2, &\eps_{24}&=5+2\sqrt6, &\eps_{28}&=8+3\sqrt7,\\
\eps_{33}&=23+4\sqrt{33}, &\eps_{40}&=3+\sqrt{10}, &\eps_{44}&=10+3\sqrt{11}, &\eps_{56}&=15+4\sqrt{14}.
\end{aligned}
\]
Also \(\eps_{60}=4+\sqrt{15}\) and \(\eps_{65}=8+\sqrt{65}\).
Theorem~\ref{thm:two-fricke} gives the exponents. For a factorization \(D_1D_2\) contributing at the principal class, the exponent of \(\eps_{D_1}\) is
\(-2c_\psi h(D_1)h(D_2)/\bigl(w(D_2)[\Cl(D):H]\bigr)\).
The odd-cardinality values are given in Theorem~\ref{thm:two-fricke}(a).

\smallskip\noindent\textit{Computing a row.}
For the first point use $D=-4M$, and for Table~\ref{tab:appendix-3} use $D=-M$. Try each divisor $D_1>1$ of $|D_K|$, put $D_2=D_K/D_1$, and retain coprime fundamental discriminants. An integer is a quadratic field discriminant if it is squarefree and $1\bmod4$ (excluding $1$), or four times a squarefree integer congruent to $2,3\bmod4$. In \eqref{eq:hhgenusprime} use the factor $d$ prime to $p$. For odd $p$ its sign is the residue $d^{(p-1)/2}\equiv\pm1\bmod p$, and at $2$ it is $(-1)^{(d^2-1)/8}$. The negative signs at $p\mid M$ select $T$. Discard empty $T$. For $M\equiv1\bmod4$ the same rule at $2$ gives $\psi(X')$, marking a factor when it is $-1$. Compute $h(D_2)$ by the reduced-form count in Section~6.

\smallskip\noindent\textit{Real class numbers.}
Put $r=\lfloor\sqrt d\rfloor$. Retain $1\le a,b\le r$ with $2a-b\le r<2a+b$, integral $c=(b^2-d)/(4a)$, and $\gcd(a,b,c)=1$. These conditions say that $\alpha=(b+\sqrt d)/(2a)>1$ and $-1<\bar\alpha<0$. From each unvisited pair mark successive pairs $(a',b')$ until it returns, where $b'=2a\lfloor(r+b)/(2a)\rfloor-b$ and $a'=(d-b'^2)/(4a)$. This is $\alpha\mapsto1/(\alpha-\lfloor\alpha\rfloor)$, whose cycles count ordinary ideal classes \cite[\S2]{lmw}. The 68 discriminants in the tables give 554 pairs. Among the tabulated $d$, class number one occurs for $d<40$ and for
\[
 41,44,53,56,57,61,69,76,77,88,89,92,93,133,161,177,209,236,253 .
\]
Class number four occurs at $145,780,1105,1365,1820$. The remaining $33$ have class number two. For example, at $145$ the cycles starting at $(1,11),(2,9),(2,11),(4,7)$ have lengths $3,3,3,5$ and exhaust the $14$ pairs.

\begin{table}[H]
\centering\small
\caption{First Fricke point \(i/\sqrt M\). All nonempty even \(T\) at the thirty known squarefree idoneal levels with \(k\ge2\). For \(M\equiv1\bmod4\), invert the factor marked \(*\) to obtain \(\varpi'_T(M)\).}\label{tab:appendix-2}
\medskip\setlength{\tabcolsep}{3pt}\renewcommand{\arraystretch}{1.08}
\begin{tabular}{@{}rll@{\hspace{1.75em}}rll@{\hspace{1.75em}}rll@{}}\toprule
\(M\)&\(T\)&\(\varpi_T(M)\) & \(M\)&\(T\)&\(\varpi_T(M)\) & \(M\)&\(T\)&\(\varpi_T(M)\)\\ \midrule
6 & \(\{2,3\}\) & \(\eps_{8}^{-1/3}\) & 130 & \(\{2,5\}\) & \(\eps_{13}^{-2}\) & 345 & \(\{3,5\}\) & \(\eps_{12}^{-1\,*}\eps_{92}^{-1}\) \\
10 & \(\{2,5\}\) & \(\eps_{5}^{-1}\) &  & \(\{2,13\}\) & \(\eps_{5}^{-6}\) &  & \(\{3,23\}\) & \(\eps_{5}^{-4\,*}\eps_{345}^{-1/2}\) \\
15 & \(\{3,5\}\) & \(\eps_{5}^{-5/3}\) &  & \(\{5,13\}\) & \(\eps_{65}^{-2}\) &  & \(\{5,23\}\) & \(\eps_{60}^{-3}\eps_{460}^{-1/3\,*}\) \\
21 & \(\{3,7\}\) & \(\eps_{12}^{-1/2}\eps_{21}^{-1/4\,*}\) & 133 & \(\{7,19\}\) & \(\eps_{76}^{-1/2}\eps_{133}^{-1/4\,*}\) & 357 & \(\{3,7\}\) & \(\eps_{17}^{-2}\eps_{357}^{-1/2\,*}\) \\
22 & \(\{2,11\}\) & \(\eps_{8}^{-1}\) & 165 & \(\{3,5\}\) & \(\eps_{5}^{-2\,*}\eps_{12}^{-2}\) &  & \(\{3,17\}\) & \(\eps_{204}^{-1}\eps_{476}^{-1/3\,*}\) \\
30 & \(\{2,3\}\) & \(\eps_{5}^{-2}\) &  & \(\{3,11\}\) & \(\eps_{44}^{-1}\eps_{165}^{-1/2\,*}\) &  & \(\{7,17\}\) & \(\eps_{12}^{-5}\eps_{28}^{-1\,*}\) \\
 & \(\{2,5\}\) & \(\eps_{40}^{-2/3}\) &  & \(\{5,11\}\) & \(\eps_{33}^{-1}\eps_{220}^{-1/3\,*}\) & 385 & \(\{5,7\}\) & \(\eps_{5}^{-4\,*}\eps_{220}^{-1}\) \\
 & \(\{3,5\}\) & \(\eps_{8}^{-2}\) & 177 & \(\{3,59\}\) & \(\eps_{177}^{-1/4}\eps_{236}^{-1/6\,*}\) &  & \(\{5,11\}\) & \(\eps_{28}^{-2}\eps_{77}^{-1\,*}\) \\
33 & \(\{3,11\}\) & \(\eps_{33}^{-1/4}\eps_{44}^{-1/6\,*}\) & 190 & \(\{2,5\}\) & \(\eps_{5}^{-6}\) &  & \(\{7,11\}\) & \(\eps_{140}^{-1\,*}\eps_{385}^{-1/2}\) \\
42 & \(\{2,3\}\) & \(\eps_{56}^{-1/3}\) &  & \(\{2,19\}\) & \(\eps_{40}^{-2}\) & 462 & \(\{2,3\}\) & \(\eps_{77}^{-2}\) \\
 & \(\{2,7\}\) & \(\eps_{21}^{-1}\) &  & \(\{5,19\}\) & \(\eps_{8}^{-8}\) &  & \(\{2,7\}\) & \(\eps_{168}^{-2}\) \\
 & \(\{3,7\}\) & \(\eps_{24}^{-1}\) & 210 & \(\{2,3\}\) & \(\eps_{21}^{-2}\) &  & \(\{2,11\}\) & \(\eps_{616}^{-2/3}\) \\
57 & \(\{3,19\}\) & \(\eps_{12}^{-1/2\,*}\eps_{57}^{-1/4}\) &  & \(\{2,5\}\) & \(\eps_{280}^{-2/3}\) &  & \(\{3,7\}\) & \(\eps_{264}^{-2}\) \\
58 & \(\{2,29\}\) & \(\eps_{29}^{-1}\) &  & \(\{2,7\}\) & \(\eps_{24}^{-2}\) &  & \(\{3,11\}\) & \(\eps_{8}^{-12}\) \\
70 & \(\{2,5\}\) & \(\eps_{8}^{-2}\) &  & \(\{3,5\}\) & \(\eps_{120}^{-2}\) &  & \(\{7,11\}\) & \(\eps_{33}^{-4}\) \\
 & \(\{2,7\}\) & \(\eps_{5}^{-4}\) &  & \(\{3,7\}\) & \(\eps_{56}^{-2}\) &  & \(\{2,3,7,11\}\) & \(\eps_{21}^{-2}\) \\
 & \(\{5,7\}\) & \(\eps_{40}^{-2}\) &  & \(\{5,7\}\) & \(\eps_{105}^{-2}\) & 1365 & \(\{3,5\}\) & \(\eps_{105}^{-2}\eps_{1820}^{-2/3\,*}\) \\
78 & \(\{2,3\}\) & \(\eps_{104}^{-2/3}\) &  & \(\{2,3,5,7\}\) & \(\eps_{5}^{-4}\) &  & \(\{3,7\}\) & \(\eps_{140}^{-4}\eps_{1365}^{-1\,*}\) \\
 & \(\{2,13\}\) & \(\eps_{13}^{-2}\) & 253 & \(\{11,23\}\) & \(\eps_{44}^{-3/2}\eps_{253}^{-1/4\,*}\) &  & \(\{3,13\}\) & \(\eps_{60}^{-2\,*}\eps_{65}^{-4}\) \\
 & \(\{3,13\}\) & \(\eps_{8}^{-4}\) & 273 & \(\{3,7\}\) & \(\eps_{12}^{-1\,*}\eps_{273}^{-1/2}\) &  & \(\{5,7\}\) & \(\eps_{12}^{-10}\eps_{13}^{-4\,*}\) \\
85 & \(\{5,17\}\) & \(\eps_{5}^{-2\,*}\eps_{17}^{-1}\) &  & \(\{3,13\}\) & \(\eps_{21}^{-1\,*}\eps_{156}^{-1}\) &  & \(\{5,13\}\) & \(\eps_{28}^{-2\,*}\eps_{273}^{-2}\) \\
93 & \(\{3,31\}\) & \(\eps_{12}^{-3/2}\eps_{93}^{-1/4\,*}\) &  & \(\{7,13\}\) & \(\eps_{13}^{-2\,*}\eps_{28}^{-2}\) &  & \(\{7,13\}\) & \(\eps_{21}^{-4\,*}\eps_{364}^{-2}\) \\
102 & \(\{2,3\}\) & \(\eps_{8}^{-2}\) & 330 & \(\{2,3\}\) & \(\eps_{5}^{-8}\) &  & \(\{3,5,7,13\}\) & \(\eps_{5}^{-4\,*}\eps_{780}^{-2}\) \\
 & \(\{2,17\}\) & \(\eps_{136}^{-2/3}\) &  & \(\{2,5\}\) & \(\eps_{165}^{-2}\) &  & &  \\
 & \(\{3,17\}\) & \(\eps_{17}^{-2}\) &  & \(\{2,11\}\) & \(\eps_{120}^{-2}\) &  & &  \\
105 & \(\{3,5\}\) & \(\eps_{60}^{-1}\eps_{140}^{-1/3\,*}\) &  & \(\{3,5\}\) & \(\eps_{33}^{-2}\) &  & &  \\
 & \(\{3,7\}\) & \(\eps_{5}^{-2\,*}\eps_{105}^{-1/2}\) &  & \(\{3,11\}\) & \(\eps_{24}^{-4}\) &  & &  \\
 & \(\{5,7\}\) & \(\eps_{12}^{-1\,*}\eps_{28}^{-1}\) &  & \(\{5,11\}\) & \(\eps_{88}^{-2}\) &  & &  \\
 & &  &  & \(\{2,3,5,11\}\) & \(\eps_{440}^{-2/3}\) &  & &  \\
\bottomrule\end{tabular}
\end{table}
\clearpage\normalsize
\begin{table}[H]
\centering\normalsize
\caption{Second Fricke point \((1+i/\sqrt M)/2\). All nonempty even \(T\) at the twenty-four known squarefree levels \(M\equiv3\bmod4\) with \(k\ge2\) and one class per genus for \(-M\).}\label{tab:appendix-3}
\medskip\setlength{\tabcolsep}{3pt}\renewcommand{\arraystretch}{1.00}
\begin{tabular}{@{}rll@{\hspace{1.75em}}rll@{}}\toprule
\(M\)&\(T\)&\(\varpi'_T(M)\) & \(M\)&\(T\)&\(\varpi'_T(M)\)\\ \midrule
15 & \(\{3,5\}\) & \(\eps_{5}^{-1/3}\) & 795 & \(\{3,5\}\) & \(\eps_{53}^{-2}\) \\
35 & \(\{5,7\}\) & \(\eps_{5}^{-1}\) &  & \(\{3,53\}\) & \(\eps_{5}^{-10}\) \\
51 & \(\{3,17\}\) & \(\eps_{17}^{-1/3}\) &  & \(\{5,53\}\) & \(\eps_{265}^{-2/3}\) \\
91 & \(\{7,13\}\) & \(\eps_{13}^{-1}\) & 1155 & \(\{3,5\}\) & \(\eps_{165}^{-2}\) \\
115 & \(\{5,23\}\) & \(\eps_{5}^{-3}\) &  & \(\{3,7\}\) & \(\eps_{5}^{-12}\) \\
123 & \(\{3,41\}\) & \(\eps_{41}^{-1/3}\) &  & \(\{3,11\}\) & \(\eps_{21}^{-4}\) \\
187 & \(\{11,17\}\) & \(\eps_{17}^{-1}\) &  & \(\{5,7\}\) & \(\eps_{33}^{-2}\) \\
195 & \(\{3,5\}\) & \(\eps_{65}^{-2/3}\) &  & \(\{5,11\}\) & \(\eps_{385}^{-2/3}\) \\
 & \(\{3,13\}\) & \(\eps_{5}^{-4}\) &  & \(\{7,11\}\) & \(\eps_{105}^{-2}\) \\
 & \(\{5,13\}\) & \(\eps_{13}^{-2}\) &  & \(\{3,5,7,11\}\) & \(\eps_{77}^{-2}\) \\
235 & \(\{5,47\}\) & \(\eps_{5}^{-5}\) & 1435 & \(\{5,7\}\) & \(\eps_{5}^{-14}\) \\
267 & \(\{3,89\}\) & \(\eps_{89}^{-1/3}\) &  & \(\{5,41\}\) & \(\eps_{205}^{-2}\) \\
403 & \(\{13,31\}\) & \(\eps_{13}^{-3}\) &  & \(\{7,41\}\) & \(\eps_{41}^{-2}\) \\
427 & \(\{7,61\}\) & \(\eps_{61}^{-1}\) & 1995 & \(\{3,5\}\) & \(\eps_{665}^{-2/3}\) \\
435 & \(\{3,5\}\) & \(\eps_{5}^{-6}\) &  & \(\{3,7\}\) & \(\eps_{5}^{-16}\) \\
 & \(\{3,29\}\) & \(\eps_{29}^{-2}\) &  & \(\{3,19\}\) & \(\eps_{105}^{-2}\) \\
 & \(\{5,29\}\) & \(\eps_{145}^{-4/3}\) &  & \(\{5,7\}\) & \(\eps_{133}^{-2}\) \\
483 & \(\{3,7\}\) & \(\eps_{69}^{-1}\) &  & \(\{5,19\}\) & \(\eps_{57}^{-2}\) \\
 & \(\{3,23\}\) & \(\eps_{161}^{-1/3}\) &  & \(\{7,19\}\) & \(\eps_{21}^{-8}\) \\
 & \(\{7,23\}\) & \(\eps_{21}^{-3}\) &  & \(\{3,5,7,19\}\) & \(\eps_{285}^{-2}\) \\
555 & \(\{3,5\}\) & \(\eps_{185}^{-2/3}\) & 3003 & \(\{3,7\}\) & \(\eps_{77}^{-4}\) \\
 & \(\{3,37\}\) & \(\eps_{5}^{-8}\) &  & \(\{3,11\}\) & \(\eps_{1001}^{-2/3}\) \\
 & \(\{5,37\}\) & \(\eps_{37}^{-2}\) &  & \(\{3,13\}\) & \(\eps_{429}^{-2}\) \\
595 & \(\{5,7\}\) & \(\eps_{17}^{-2}\) &  & \(\{7,11\}\) & \(\eps_{13}^{-12}\) \\
 & \(\{5,17\}\) & \(\eps_{85}^{-2}\) &  & \(\{7,13\}\) & \(\eps_{273}^{-2}\) \\
 & \(\{7,17\}\) & \(\eps_{5}^{-10}\) &  & \(\{11,13\}\) & \(\eps_{21}^{-10}\) \\
627 & \(\{3,11\}\) & \(\eps_{209}^{-1/3}\) &  & \(\{3,7,11,13\}\) & \(\eps_{33}^{-2}\) \\
 & \(\{3,19\}\) & \(\eps_{33}^{-1}\) & 3315 & \(\{3,5\}\) & \(\eps_{17}^{-4}\) \\
 & \(\{11,19\}\) & \(\eps_{57}^{-1}\) &  & \(\{3,13\}\) & \(\eps_{221}^{-4}\) \\
715 & \(\{5,11\}\) & \(\eps_{13}^{-4}\) &  & \(\{3,17\}\) & \(\eps_{65}^{-4}\) \\
 & \(\{5,13\}\) & \(\eps_{5}^{-10}\) &  & \(\{5,13\}\) & \(\eps_{13}^{-12}\) \\
 & \(\{11,13\}\) & \(\eps_{65}^{-2}\) &  & \(\{5,17\}\) & \(\eps_{1105}^{-4/3}\) \\
 & &  &  & \(\{13,17\}\) & \(\eps_{85}^{-8}\) \\
 & &  &  & \(\{3,5,13,17\}\) & \(\eps_{5}^{-16}\) \\
\bottomrule\end{tabular}
\end{table}

\noindent\textit{Two sample entries.}
At $M=15$ the only factorization is $5(-3)$, with $T=\{3,5\}$ and $h(5)=h(-3)=1$. Both symbols at $2$ are $-1$, so $c_\psi=5$ at $D=-60$ and $1$ at $D=-15$. Since $w(-3)=6$, the exponents are $-5/3,-1/3$. At $M=21$ the factorizations $12(-7),21(-4)$ for $T=\{3,7\}$ have signs $+1,-1$ at $2$. All four class numbers are one and $[\Cl(-84):H]=2$, giving exponents $-1/2,-1/4$ at the first point and $-1/2,+1/4$ at the second. Each $k$-prime level contributes $2^{k-1}-1$ nonempty even sets, giving $80$ first-point and $99$ second-point entries including the starred rule.

\smallskip\noindent\textit{Fundamental units.}
For each $d$ test $dy^2-4,dy^2+4$ in that order at $y=1,2,\ldots$ until a positive square $x^2$ occurs. Then $\eps_d=(x+y\sqrt d)/2$. Indeed the norm is $\pm1$, and $|\eps-y\sqrt d|=1/\eps<1$ for any such unit $\eps>1$. As $\sqrt d>2$, increasing $y$ orders the candidates. The largest first $y$ among the 68 is $67064$ at $d=1001$, with $x=2121810$. The finite Dirichlet identity
\[
 h(d)\log\eps_d=-\sum_{1\le a<d/2}\chi_d(a)\log(2\sin(\pi a/d))
\]
also follows by taking the Abel limit of the cosine series for $-\log(2\sin(\pi a/d))$. Pair $a,d-a$ and use $\sum_{a\bmod d}\chi_d(a)e^{2\pi ian/d}=\sqrt d\,\chi_d(n)$ and $L(1,\chi_d)=2h(d)\log\eps_d/\sqrt d$.

\end{document}